\documentclass[a4paper] {article}
[15pt]

   \usepackage[top=2.5cm, bottom=2.5cm, left=3.5cm, right= 3.5cm]{geometry}

\makeatletter
\renewcommand*\l@section{\@dottedtocline{1}{1.5em}{2.3em}}
\makeatother

\usepackage{amsfonts}
\usepackage{amssymb}
\usepackage[T1]{fontenc}

\usepackage{tikz}
\usetikzlibrary{calc}

\usepackage{CJK}
\usepackage{amsmath}
 \usepackage{algorithmicx}
 \usepackage[ruled]{algorithm}
 \usepackage{algpseudocode}

\usepackage{amsfonts}
\usepackage{amssymb}
\usepackage{amsthm}
\usepackage{amssymb}
\usepackage{enumerate}
\usepackage[calc]{picture}
\usepackage[all,cmtip]{xy}

\usepackage[mathscr]{eucal}
\usepackage{eqlist}

\usepackage{color}
\usepackage{abstract}
\usepackage[T1]{fontenc}

\theoremstyle{plain}
\newtheorem{theorem}{Theorem}
\newtheorem{proposition}[theorem]{Proposition}
\newtheorem{lemma}[theorem]{Lemma}

\newtheorem{example}[theorem]{Example}
\newtheorem{corollary}[theorem]{Corollary}
\newtheorem{remark}[theorem]{Remark}

\theoremstyle{definition}
\newtheorem{definition}[theorem]{Definition}

\usepackage{etoolbox}
\newtheoremstyle{myrem}
 {3pt}
 {3pt}
 {\normalsize}
 { }
 {\itshape}
 {:}
 { }
 {}

 \theoremstyle{myrem}
 \appto\remark{\leftskip\parindent}
 \appto\remark{\rightskip\parindent}

\numberwithin{equation}{section}
\numberwithin{theorem}{section}

\begin{document}

\begin{center}
{\Large {\textbf {Witten-Morse functions and Morse inequalities on digraphs}}}
 \vspace{0.58cm} \\
\bigskip
Yong Lin$^{*}$, Chong Wang$^{\dag}$

\bigskip

\bigskip

    \parbox{24cc}{{\small
{\textbf{Abstract}.}
In this paper, we prove that discrete Morse functions on  digraphs are flat Witten-Morse functions and   Witten complexes of transitive digraphs approach to Morse complexes. We construct a chain complex consisting of the formal linear combinations of paths  which are not only critical paths of the transitive closure but also allowed elementary paths of the digraph, and prove that the homology of the new chain complex is isomorphic to the path homology. On the basis of the above results, we give the Morse inequalities on digraphs. 
}}
 \end{center}

 \vspace{1cc}


\footnotetext{

\rule{2.4mm}{0mm}{\bf Keywords and Phrases.}  path homology,  Witten-Morse function, Morse inequalities, transitive closure, digraph.

}
\section{Introduction}
Digraphs are  important topological models  in complex networks. A digraph $G$  is determined by a finite  set $V$ and a non-empty subset $E$ of $ V\times V\setminus\{\mathrm{diag}\}$. $V$ is called the vertex set of $G$ and $E$ is called the directed edge set of $G$. For vertices $u,v \in V$, the pair $(u,v)\in E$ is denoted as $u \rightarrow v$. $G$ is called {\it transitive}  if for any two directed edges  $u\to v$ and $v\to w$ of G, there is a directed edge  $u\to w$  of $G$.
The transitive closure of $G$ is the smallest transitive digraph containing $G$, which is denoted as $\bar{G}$ in this paper if there is no ambiguity.

Let $R$ be the real numbers. For any  integer $n\geq 0$, an {\it elementary} $n$-path is a sequence $v_0 v_1\cdots v_n$ of $n+1$ vertices in $V$. Let $\Lambda_n(V)$ be the linear space consisting of all the formal linear combinations of the $n$-paths on $V$. The $i$-th face map is defined as the $R$-linear map
 \begin{eqnarray*}
d_i : \Lambda_n(V)\longrightarrow \Lambda_{n-1}(V)
 \end{eqnarray*}
 which sends $v_0v_1\cdots v_n$ to $v_0\cdots \hat{v_i}\cdots v_n$, where $\hat{v_i}$  means omission of the vertex $v_i$. Let $\partial_n=\sum\limits_{i=0}^n (-1)^i d_i$. Then $\partial_n$ is an $R$-linear map from $\Lambda_n(V)$ to $\Lambda_{n-1}(V)$ satisfying $\partial_{n}\partial_{n+1}=0$ for each $n\geq 0$ (cf. \cite{9,10,3,4,2,yau2,1}).  Hence $\{\Lambda_n(V),\partial_n\}_{n\geq 0}$ is a chain complex.

An {\it  allowed elementary $n$-path} on $G$ is a $n$-path $v_0 v_1\ldots v_n$ on $V$ such that $v_{i-1}\to v_i$ is a directed edge of $G$ and  $v_{i-1}\not=v_i$ for each $1\leq i \leq n$. Let $P_n(G)$ be the linear space consisting of all the formal linear combinations of allowed elementary $n$-paths on $G$. Then $P_n(G)$ is a  subspace of $\Lambda_n(V)$, whereas the image  of an allowed elementary path  under the boundary operator $\partial$ does not have to be allowed. Consider the space $\Omega_n(G)$ formed by all the  linear combinations of
the $\partial$-invariant $n$-paths in $P_n(G)$. Obviously, $\Omega_n(G)$ is a subspace of $P_n(G)$. The path homology of $G$ is defined as the homology of chain complex $\{\Omega_n(G),\partial_n\}_{n\geq 0}$ and denoted as $H_*(G;R)$. That is,
\begin{eqnarray*}\label{eq-99}
H_m(G;R)=H_m(\{\Omega_n(G),\partial_n\}_{n\geq 0}),\quad m\geq 0.
\end{eqnarray*}

Morse theory can simplify the calculation of homology groups. Using Morse theory, one can determine the cell decomposition of manifolds by studying the negative inertia index of Hessian matrix of Morse functions at the critical points, so as to characterize the homology groups.  In 1925, M. Morse first invented the method of Morse theory (cf. \cite{morse}). In 1963, J.W. Milnor combed, studied and developed Morse's method, and  Morse theory was given in \cite{milnor}.  Since then,  there have been numerous  researches on Morse theory (cf. \cite{floer,schwarz}, etc). In recent years, Morse theory has been applied to cell complexes, simplicial complexes, graphs and other combinatorial objects, and discrete Morse theory  has gradually become a hot research topic (cf. \cite{ayala1,ayala2,ayala3,ayala4,forman1,forman2, forman3, witten}). 

It is well known that the homology groups of simplicial complexes or cell complexes  can be characterized by chain complexes made of the linear combinations of critical simplices. Inspired by this, in this paper, based on \cite{lin,wang}, we further study the properties of discrete Morse functions on digraphs and critical paths on transitive digraphs, characterize the path homology groups of digraphs with  chain
complex consisting of the formal linear combinations of paths which are not only critical paths of the transitive closure but also allowed elementary paths of the digraph, prove that Witten complexes of transitive digraphs  approach to Morse complexes, which is not necessarily true for general digraphs.  

Let $G$ be a digraph and $f: V(G)\longrightarrow [0,+\infty)$  a discrete Morse function on $G$ as defined in \cite{wang} and Definition~\ref{def-1.1}. By \cite[Definition~6.1]{forman1}, the (algebraic) discrete gradient vector field of $f$  is defined as an $R$-linear map $\mathrm{grad}f: P_n(G)\rightarrow P_{n+1}(G)$ such that for any allowed elementary $n$-path $\alpha$ on $G$,
\begin{eqnarray*}
(\mathrm{grad}f)(\alpha)=-\langle\partial\gamma,\alpha\rangle\gamma,
\end{eqnarray*}
where $\gamma>\alpha$ and $f(\gamma)=f(\alpha)$. Otherwise $(\mathrm{grad}f)(\alpha)=0$. Here $\langle, \rangle$ is  the  inner product in $\Lambda_n(V)$  (with respect to which the elementary $n$-paths are orthonormal).


Let $\bar{f}$ be a discrete Morse function on a transitive digraph and $\overline{V}=\mathrm{grad}\bar{f}$ the discrete gradient vector field on it. By \cite[Definition~6.2]{forman1}, the discrete gradient flow is denoted as
\begin{eqnarray*}
\overline{\Phi}=\mathrm{Id}+\partial \overline{V}+\overline{V}\partial.
\end{eqnarray*}
Let $\Delta_n(t)$ be the   Laplace operator with one-parameter $t$  and ${W}_n(t)$  the span of the eigenvectors of $\Delta_n(t)$ corresponding to the eigenvalues which tend to $0$ as $t\to \infty$. Denote $\mathrm{Crit}_ n(-)$ as the span of all  critical $n$-paths on ``$-$''.

The paper is organized as follows. In Section~\ref{sec-2}, we review the definition of  discrete Morse functions on digraphs and prove that discrete Morse functions on digraphs are discrete flat Witten-Morse functions in Proposition~\ref{pr-s1}.  Then we prove that Witten complexes approach to Morse complexes for transitive digraphs in Section~\ref{se-gg}. Furthermore, we study the path homology of general digraphs in Section~\ref{se-hh} which is divided into two subsections. In Subsection~\ref{sec-3}, we give some properties of transitive digraphs. Particularly, we characterize $\overline{\Phi}$-invariant space with  critical paths in Proposition~\ref{prop-11}. Let $G$ be a digraph and $\bar{G}$ the transitive closure of $G$. Suppose $\Omega_*(G)$ is $\overline{V}$-invariant ($\overline{V}(\Omega(G))\subseteq \Omega(G)$). Then
\begin{eqnarray*}
H_m(G;R)\cong  H_m\big(\{R(\alpha+\overline{V}\partial(\alpha))\cap \Omega_n(G),\partial_n\}_{n\geq 0}\big)
\end{eqnarray*}
where $\alpha\in\mathrm{Crit}(\bar{G})$. This is proved in Theorem~\ref{th-999}.

Moreover, in Subsection~\ref{sec-4}, we give a description of path homology of digraphs by homology of  a chain complex which is related to critical sets of the transitive closure of $G$ in Corollary~\ref{co-aa}. That is, if $\Omega_*(G)$ is $\overline{V}$-invariant and $\overline{\Phi}(\alpha)\in \Omega(G)$  for any $\alpha\in \mathrm{Crit}(\bar{G})\cap P(G)$, then
\begin{eqnarray*}
H_m(\{\mathrm{Crit}_n(\bar{G})\cap P_n(G),\tilde{\partial}_n\}_{n\geq 0})\cong H_m(G;R)
\end{eqnarray*}
where $\tilde{\partial}=(\overline{\Phi}^{\infty})^{-1}\circ\partial \circ \overline{\Phi}^{\infty}$ and $\overline{\Phi}^{\infty}$ is the stabilization map of $\overline{\Phi}$. 

Finally,  in Section~\ref{se-99}, we give the Morse inequalities on digraphs.

\section{Preliminaries}\label{sec-2}
In this section, we mainly review the definition of   discrete Morse functions on digraphs and prove that discrete Morse functions on digraphs are flat Witten-Morse functions.
\smallskip

For any allowed elementary paths $\alpha$ and $\beta$, if $\beta$ can be obtained from $\alpha$ by removing some vertices, then we write $\alpha>\beta$ or $\beta<\alpha$.
\begin{definition}(cf. \cite{wang})\label{def-1.1}
 A  map $f: V(G)\longrightarrow [0,+\infty)$ is called a  {\it discrete Morse function} on $G$, if for any allowed elementary path $\alpha=v_0v_1\cdots v_n$ on $G$, both of the followings hold:
 \begin{quote}
\begin{enumerate}[(i).]
\item
$\#\Big\{\gamma^{(n+1)}>\alpha^{(n)}\mid f(\gamma)=f(\alpha)\Big\}\leq 1$;
\item
$\#\Big\{\beta^{(n-1)}<\alpha^{(n)}\mid f(\beta)=f(\alpha)\Big\}\leq 1$.
\end{enumerate}
\end{quote}
where
\begin{eqnarray*}\label{eq-1.1}
f(\alpha)=f(v_0v_1\cdots v_n)= \sum_{i=0}^n f(v_i).
\end{eqnarray*}
\end{definition}
For an allowed elementary path $\alpha$, if in both (i) and (ii), the inequalities hold strictly, then $\alpha$ is called {\it critical}. Precisely,
\begin{definition}\label{def-1.2}
An allowed elementary $n$-path $\gamma^{(n)}$ is called {\it critical}, if both of the followings hold:
\begin{quote}
\begin{enumerate}[(i)']
\item
$\#\Big\{\beta^{(n-1)}<\alpha^{(n)}\mid f(\beta)=f(\alpha)\Big\}=0$,
\item
$\#\Big\{\gamma^{(n+1)}>\alpha^{(n)}\mid f(\gamma)=f(\alpha)\Big\}=0$.
\end{enumerate}
\end{quote}
\end{definition}

It follows from Definition~\ref{def-1.2} that an allowed elementary  $p$-path  is not critical if and only if
either of the following conditions holds
\begin{quote}
\begin{enumerate}[(i)'']
\item
there exists $\beta^{(n-1)}<\alpha^{(n)}$ such that $f(\beta)=f(\alpha)$;
\item
there exists $\gamma^{(n+1)}>\alpha^{(n)}$ such that $f(\gamma)=f(\alpha)$.
\end{enumerate}
\end{quote}

A directed loop on $G$ is an allowed elementary path $v_0 v_1 \cdots v_n v_0$, $n \geq 1$.
\begin{lemma}(cf. \cite[Lemma~2.4]{lin})\label{le-1.14}
Let $G$ be a digraph and $f$  a discrete Morse function on $G$. Let $\alpha=v_0v_1\cdots v_n v_0$ be a directed loop. Then for each $0\leq i\leq n$, $f(v_i)>0$.
\end{lemma}

\begin{lemma}(cf. \cite[Lemma~2.5]{lin})\label{le-1.13}
Let $G$ be a digraph and $f$  a discrete Morse function on $G$ as defined in Definition~\ref{def-1.1}. Then for any allowed elementary path in $G$, there exists at most one index such that the corresponding  vertex is with zero value.
\end{lemma}

\begin{lemma}\label{le-1.24}
Let $f$ be a discrete Morse function on digraph $G$. Then for  any allowed elementary path $\alpha=v_0v_1\cdots v_{n}$ in $G$,   (i)'' and (ii)'' cannot both be true.
\end{lemma}
\begin{proof}
Suppose to the contrary. By (i)'',  there exists an allowed elementary $(n-1)$-path $\beta$  such that $\beta<\alpha$ and $f(\beta)=f(\alpha)$. Hence, there exists some $0\leq i \leq n$ such that $f(v_i)=0$. By (ii)'',  there exists an allowed elementary $(n+1)$-path $\gamma$ such that $f(\gamma)=f(\alpha)$. Hence, there exists a vertex $u\in V(G)$ with $f(u)=0$ such that $\gamma=v_0\cdots v_j uv_{j+1}\cdots v_n$. We assert that $u\not= v_i$. Suppose to the contrary.  Since $\gamma$ is allowed, it follows that $u$ and $v_i$ are not adjacent. Hence there exists a directed loop in which  $u$ is a vertex. This contradicts Lemma~\ref{le-1.14}.  Therefore, there are two distinct  vertices   with zero value in $\gamma$  which contradicts Lemma~\ref{le-1.13}.

The lemma follows.
\end{proof}

\begin{definition}(cf. \cite[Definition~0.6]{witten})\label{def-witten}
A function $f: V(G)\longrightarrow [0,+\infty)$ is called a  {\it discrete Witten-Morse function} on $G$ if, for any allowed elementary path $\alpha$,
\begin{quote}
\begin{enumerate}[(i)]
\item
$f(\alpha)<\mathrm{average}\{f(\gamma_1), f(\gamma_2)\}$ where $\gamma_1>\alpha$, $\gamma_2>\alpha$ and $\gamma_1\not=\gamma_2$;
\item
$f(\alpha)>\mathrm{average}\{f(\beta_1), f(\beta_2)\}$ where $\beta_1<\alpha$, $\beta_2<\alpha$ and $\beta_1\not= \beta_2$.
\end{enumerate}
\end{quote}
\end{definition}
Note that each Witten-Morse function is, in fact, a Morse function.

\begin{definition}(cf. \cite[Definition~0.7]{witten})\label{def-flat}
A discrete Witten-Morse funtion is {\it flat} if for any allowed elementary path $\alpha$,
\begin{quote}
\begin{enumerate}[(i)]
\item
$f(\alpha)\leq \mathrm{min}\{f(\gamma_1), f(\gamma_2)\}$ where $\gamma_1>\alpha$, $\gamma_2>\alpha$ and $\gamma_1\not= \gamma_2$;
\item
$f(\alpha)\geq \mathrm{max}\{f(\beta_1), f(\beta_2)\}$ where $\beta_1<\alpha$, $\beta_2<\alpha$ and $\beta_1\not= \beta_2$.
\end{enumerate}
\end{quote}
\end{definition}

\begin{proposition}\label{pr-s1}
Let $G$ be a digraph and $f: V(G)\longrightarrow [0,+\infty)$ a discrete Morse function on $G$. Then $f$ is a discrete flat Witten-Morse function.
\end{proposition}
\begin{proof}
Let $\alpha$ be an arbitrary allowed elementary path on $G$. Consider the following cases.

{\sc Case~1}. $\alpha$ is critical. Then by Definition~\ref{def-1.2}, we have that $f(\alpha)<f(\gamma)$ for any $\gamma>\alpha$ and $f(\alpha)>f(\beta)$ for any $\beta<\alpha$.  Hence,
\begin{eqnarray*}
f(\alpha)< \mathrm{min}\{f(\gamma_1), f(\gamma_2)\}
\end{eqnarray*}
where $\gamma_1>\alpha$ and $\gamma_2>\alpha$, and
\begin{eqnarray*}
f(\alpha)> \mathrm{max}\{f(\beta_1), f(\beta_2)\}
\end{eqnarray*}
where $\beta_1<\alpha$ and $\beta_2<\alpha$.

{\sc Case~2}. $\alpha$ is not critical. 

{\sc Subcase~2.1}. (i)'' holds for $\alpha$. Then there exists an unique allowed elementary path $\beta<\alpha$ such that $f(\beta)=f(\alpha)$.
Let $\beta_1<\alpha$, $\beta_2<\alpha$ and $\beta_1\not= \beta_2$. Then
\begin{eqnarray*}
\left\{
\begin{array}{cc}
f(\beta_1)<f(\alpha), f(\beta_2)<f(\alpha), & \text{ if  } \beta_1\not= \beta  ~\text{and}~ \beta_2\not= \beta; \\
f(\beta_1)=f(\alpha), f(\beta_2)<f(\alpha),   & \text{ if } \beta_1=\beta ~\text{and}~ \beta_2\not= \beta;\\
f(\beta_1)<f(\alpha), f(\beta_2)=f(\alpha),   & \text{ if } \beta_1\not=\beta ~\text{and}~ \beta_2=\beta.
\end{array}
\right.
\end{eqnarray*}
Hence,
\begin{eqnarray*}
f(\alpha)\geq \mathrm{max}\{f(\beta_1), f(\beta_2)\}
\end{eqnarray*}
where $\beta_1<\alpha$, $\beta_2<\alpha$ and $\beta_1\not= \beta_2$.

By Lemma~\ref{le-1.14}, (ii)'' does not hold for $\alpha$. Then for any $\gamma_1>\alpha$, $\gamma_2>\alpha$ and $\gamma_1\not= \gamma_2$, we have  that $f(\gamma_1)>f(\alpha)$ and $f(\gamma_2)>f(\alpha)$. Hence,
\begin{eqnarray*}
f(\alpha)< \mathrm{min}\{f(\gamma_1), f(\gamma_2)\}
\end{eqnarray*}
where $\gamma_1>\alpha$, $\gamma_2>\alpha$ and $\gamma_1\not=\gamma_2$.

{\sc Subcase~2.2}. (ii)'' holds for $\alpha$. Then there exists an unique allowed elementary path $\gamma>\alpha$ such that $f(\gamma)=f(\alpha)$.
Let $\gamma_1>\alpha$, $\gamma_2>\alpha$ and $\gamma_1\not= \gamma_2$. Similarly,

\begin{eqnarray*}
\left\{
\begin{array}{cc}
f(\gamma_1)>f(\alpha), f(\gamma_2)>f(\alpha), & \text{ if  } \gamma_1\not= \gamma  ~\text{and}~ \gamma_2\not= \gamma; \\
f(\gamma_1)=f(\alpha), f(\gamma_2)>f(\alpha),   & \text{ if } \gamma_1=\gamma ~\text{and}~ \gamma_2\not= \gamma;\\
f(\gamma_1)>f(\alpha), f(\gamma_2)=f(\alpha),   & \text{ if } \gamma_1\not=\gamma ~\text{and}~ \gamma_2=\gamma.
\end{array}
\right.
\end{eqnarray*}
Hence,
\begin{eqnarray*}
f(\alpha)\leq \mathrm{min}\{f(\gamma_1), f(\gamma_2)\}
\end{eqnarray*}
where $\gamma_1>\alpha$, $\gamma_2>\alpha$ and $\gamma_1\not= \gamma_2$.

By Lemma~\ref{le-1.14}, (i)'' does not hold for $\alpha$. Then for any $\beta_1>\alpha$, $\beta_2>\alpha$ and $\beta_1\not= \beta_2$, we have  that $f(\beta_1)<f(\alpha)$ and $f(\beta_2)<f(\alpha)$. Hence,
\begin{eqnarray*}
f(\alpha)>\mathrm{max}\{f(\beta_1), f(\beta_2)\}
\end{eqnarray*}
where $\beta_1<\alpha$, $\beta_2<\alpha$ and $\beta_1\not=\beta_2$.


Combining Case 1 and Case 2, the assertion is proved.
\end{proof}
\begin{remark}
The key of Proposition~\ref{pr-s1} is the definition of discrete Morse functions on digraphs. Let $f$ be a discrete Morse function on digraph $G$. Then $f(\beta)\leq f(\alpha)$ for any $\beta<\alpha$ and $f(\gamma)\geq f(\alpha)$ for any $\gamma>\alpha$. Meanwhile, for any allowed elementary path, (i)'' and (ii)'' can not both be true.
\end{remark}
Next, we consider the {\bf equivalent} discrete Morse functions on digraphs.
\begin{definition}\label{def-0}
Let $G$ be a digraph and $f: V(G)\longrightarrow [0,+\infty)$ a discrete Morse function on $G$. 
 The set  of all vertices $v\in V(G)$ such that $f(v)=0$ is called {\it zero-point set} of $f$, denoted as $S(f)$.
\end{definition}

\begin{definition}\label{def-1}(cf. \cite[Definition~1.2]{witten})
Let $f, g$ be two discrete Morse functions on $G$. We say $f$ and $g$ are equivalent if for any $n\geq 0$ and every $\alpha^{(n)}<\gamma^{(n+1)}$,
\begin{eqnarray*}
f(\alpha)<f(\gamma)\Longleftrightarrow g(\alpha)<g(\gamma).
\end{eqnarray*}
\end{definition}

\begin{proposition}\label{le-0}
Let $f, g$ be two discrete Morse functions on $G$ such that $S(f)=S(g)$. Let $\bar{f}, \bar{g}$ be the extensions of $f, g$ on transitive closure $\bar{G}$ of $G$. Then they induce the same Morse complexes.
\end{proposition}
\begin{proof}
Let $\alpha^{(n)}$ and $\gamma^{(n+1)}$ be allowed elementary paths on $\bar{G}$ such that $\alpha<\gamma$ and  $\bar{f}(\alpha)<\bar{f}(\gamma)$. We assert that $\bar{g}(\alpha)<\bar{g}(\gamma)$. Suppose to the contrary, $\bar{g}(\gamma)=\bar{g}(\alpha)$. Then there exists a vertex $v$ in $V(\gamma)\setminus V(\alpha)$ (Consider $\alpha$ and $\gamma$ as subgraphs of $\bar{G}$) such that $\bar{g}(v)=0$.
Since $S(f)=S(g)$, it follows that $S(\bar{f})=S(\bar{g})$. Hence $\bar{f}(v)=0$ and $\bar{f}(\alpha)=\bar{f}(\gamma)$ which contradict $\bar{f}(\alpha)<\bar{f}(\gamma)$. Thus,
\begin{eqnarray*}
\bar{f}(\alpha)<\bar{f}(\gamma)\Longrightarrow \bar{g}(\alpha)<\bar{g}(\gamma).
\end{eqnarray*}

Similarly, we have that
\begin{eqnarray*}
\bar{g}(\alpha)<\bar{g}(\gamma)\Longrightarrow \bar{f}(\alpha)<\bar{f}(\gamma).
\end{eqnarray*}
Hence, by Definition~\ref{def-1}, $\bar{f}$ and  $\bar{g}$ are equivalent. Therefore, for any allowed elementary path $\alpha$ on $\bar{G}$,
\begin{eqnarray*}
\overline{V}_{\bar{f}}(\alpha)&=&\overline{V}_{\bar{g}}(\alpha),\\
\overline{\Phi}_{\bar{f}}(\alpha)&=&\overline{\Phi}_{\bar{g}}(\alpha)\\
\end{eqnarray*}
and
\begin{eqnarray*}\label{eq-le0}
\mathrm{Crit}_{\bar{f}}(\bar{G})=\mathrm{Crit}_{\bar{g}}(\bar{G})
\end{eqnarray*}
which imply that the induced Morse complexes are the same.
 \end{proof}

\bigskip
\section{Witten Complexes of Transitive  Digraphs}\label{se-gg}
In this section, we prove that  Witten complexes approach to Morse complexes for transitive digraphs.
\bigskip

Let $G$ be a transitive digraph. Similar to \cite{witten}, consider the chain complex
\begin{eqnarray}\label{eq-111}
0{\longrightarrow} \Omega_n({G}) \overset{\partial}{\longrightarrow} \Omega_{n-1}({G})\overset{\partial}{\longrightarrow}\cdots \overset{\partial}{\longrightarrow}\Omega_0({G})\longrightarrow 0.
\end{eqnarray}
Define a  chain homomorphism
\begin{eqnarray*}
e^{tf}:\Omega_n({G})\longrightarrow \Omega_n({G})
\end{eqnarray*}
by setting
\begin{eqnarray}\label{eq-l22}
e^{tf}(\alpha)=e^{tf(\alpha)}\alpha
\end{eqnarray}
for any allowed elementary path $\alpha$ on ${G}$, and  extending linearly to  $\Omega({G})$. Replace the boundary operator $\partial$ with
\begin{eqnarray*}
\partial_t=e^{tf}\partial e^{-tf}.
\end{eqnarray*}
Then
\begin{eqnarray*}
\partial_t(\alpha)&=& e^{tf}\partial e^{-tf}(\alpha)\\
&=&e^{tf}\partial e^{-tf(\alpha)}\alpha\\
&=&e^{-tf(\alpha)}e^{tf}(\partial \alpha)\\
&=&\sum\limits_{\beta<\alpha,\beta\in \Omega({G})}e^{t[f(\beta)-f(\alpha)]}\beta.
\end{eqnarray*}
Hence, $\partial_t(\alpha)\in \Omega({G})$ which implies that
\begin{eqnarray}\label{eq-23}
0{\longrightarrow} \Omega_n({G})\overset{\partial_{t}}{\longrightarrow}\Omega_{n-1}({G})\overset{\partial_{t}}{\longrightarrow}\cdots \overset{\partial_{t}}{\longrightarrow}\Omega_0({G})\longrightarrow 0
\end{eqnarray}
is still a chain complex.  Moreover, by a similar argument to \cite[Section~5.1, P.54]{supersymmetry}, we have that
\begin{proposition}\label{prop-32}
Let $G$ be a transitive digraph. Then for each $t\in  R$, the complexes (\ref{eq-111}) and (\ref{eq-23}) have the same path homology. That is,
\begin{eqnarray*}
 H_m(\{\Omega_n({G}),\partial_n\}_{n\geq 0})\cong H_m(\{\Omega_n({G}),\partial_t\}_{n\geq 0}).
\end{eqnarray*}
\end{proposition}
\begin{proof}
Note that $\Omega(G)=P(G)$ for transitive digraphs. For any $x\in \mathrm{Ker}\partial$, under the map (\ref{eq-l22}), we have that
\begin{eqnarray*}
\partial_te^{tf}(x)&=&e^{tf}\partial(x)\\
&=&e^{tf}(\partial x)\\
&=&0.
\end{eqnarray*}
That is, $e^{tf}(x)\in\mathrm{Ker}\partial_t$. And if $x=(\partial y)\in\mathrm{Im}\partial$, then

\begin{eqnarray*}
\partial_te^{tf}(y)&=&e^{tf}\partial(y)\\
&=&e^{tf}(x).
\end{eqnarray*}
Hence,  $e^{tf}(x)\in\mathrm{Im}\partial_t$. 

Therefore,  the invertible map of (\ref{eq-l2}) maps $\partial$-invariant paths which are closed but not exact in the usual sense to $\partial$-invariant paths which are closed but not exact in the sense of $\partial_t$.

The proposition is proved.

\end{proof}

Let
\begin{eqnarray*}
\Delta_n(t)=\partial_t\partial_t^{*}+\partial_t^{*}\partial_t
\end{eqnarray*}
be  the Laplace operator induced by $\partial_t$ where $\partial_t^{*}$ is the adjoint of $\partial_t$ with respect to the inner product on the chain spaces $\Lambda_*(V)$ such that all paths are orthonormal.
Then by \cite[ Section~3.1]{torsion},
\begin{eqnarray*}
\mathrm{Ker}(\Delta_n(t))\cong H_m(\{\Omega_n(G),\partial_t\}_{n\geq 0}).
\end{eqnarray*}
Hence, by Proposition~\ref{prop-32},
\begin{eqnarray}\label{eq-001}
\mathrm{Ker}(\Delta_n(t))\cong H_m(\{\Omega_n(G),\partial_n\}_{n\geq 0}).
\end{eqnarray}

\smallskip

Denote ${W}_n(t)$ as the span of the eigenvectors of $\Delta_n(t)$ corresponding to the eigenvalues which tend to $0$ as $t\to \infty$. Since $\Delta(t)\partial_t=\partial_t\Delta(t)$, $\partial_t$ preserves the eigenspaces.
The Witten complex is defined as %
\begin{eqnarray*}\label{eq-hh}
0{\longrightarrow} W_n(t)\overset{\partial_{t}}{\longrightarrow}W_{n-1}(t)\overset{\partial_{t}}{\longrightarrow}\cdots \overset{\partial_{t}}{\longrightarrow}W_0(t)\longrightarrow 0.
\end{eqnarray*}
Let $\mathrm{Crit}_ n({G})$ be the span of the  critical $n$-paths on $G$. We have the following theorem.
\begin{theorem}\label{th-40}
Let $G$ be a transitive digraph and $f$ a discrete Morse function on $G$.  Then
\begin{eqnarray*}\label{eq-41}
\lim\limits_{t\to \infty}{{W}_n(t)}=\mathrm{Crit}_n({G}).
\end{eqnarray*}
\end{theorem}
\begin{proof}
Since $G$ is  transitive, $P_n(G)=\Omega_n(G)$ for each $n\geq 0$. By \cite[Theorem~2.1]{witten} and  Proposition~\ref{pr-s1}, we have that
\begin{eqnarray*}
\Delta_n(t)\alpha&=&[\sum_{\beta<\alpha}{\langle\partial \alpha, \beta\rangle}^2e^{2t(f(\beta)-f(\alpha))}+\sum_{\gamma>\alpha}{\langle\partial \gamma, \alpha\rangle}^2e^{2t(f(\alpha)-f(\gamma))}]\alpha+O(e^{-tc})
\end{eqnarray*}
for some $c>0$, where $\gamma,\alpha, \beta$ are allowed elementary paths on $G$. Hence, if and only if $\alpha$ is critical, the eigenvalues of $\Delta_n(t)$
\begin{eqnarray*}\label{eq-s3}
\langle\Delta_n(t)\alpha,\alpha\rangle=\sum_{\beta<\alpha}{\langle\partial \alpha, \beta\rangle}^2e^{2t(f(\beta)-f(\alpha))}+\sum_{\gamma>\alpha}{\langle\partial \gamma, \alpha\rangle}^2e^{2t(f(\alpha)-f(\gamma))}
\end{eqnarray*}
tend to $0$ as $t\to \infty$.

The theorem is proved.
\end{proof}
\begin{corollary}\label{co-ff}
Let $G$ be a transitive digraph. Then Witten complex $\{W_n(t),\partial_t\}_{n\geq 0}$ approaches to the complex $\{\mathrm{Crit}_n(G), \tilde\partial_n\}_{n\geq 0}$.
\end{corollary}
\begin{proof}
By \cite[Theorem~2.1]{lin},
\begin{eqnarray*}
 H_m(\{\mathrm{Crit_n(G)},\tilde\partial_n\}_{n\geq 0}) \cong H_{m}(\{\Omega_n(G),\partial_n\}_{n\geq 0}).
\end{eqnarray*}
By (\ref{eq-001}), for all $t$,
\begin{eqnarray*}
H_m(\{W_n(t),\partial_t\}_{n\geq 0})\cong H_{m}(\{\Omega_n(G),\partial_n\}_{n\geq 0}).
\end{eqnarray*}
Therefore, by Theorem~\ref{th-40}, the assertion is  followed.
\end{proof}

\bigskip
Note that  for general digraph $G$, the image of each $\partial$-invariant element $x\in \Omega_n(G)$ under $\partial_t$ may be not in $\Omega_{n-1}(G)$. This implies that $\{\Omega(G),\partial_t\}$  is not a chain complex in general. For example,
\begin{example}\label{ex-b}
Let $G$ be a square with vertex set $V=\{v_0,v_1,v_2,v_3\}$ and directed edge set $E=\{v_0v_1,v_0v_2,v_1v_3,v_2v_3\}$. Then
\begin{eqnarray*}
\Omega(G)=\{v_0,v_1,v_2,v_3,v_0v_1,v_0v_2,v_1v_3,v_2v_3,v_0v_1v_3-v_0v_2v_3\}
\end{eqnarray*}
and
\begin{eqnarray}
\partial_t(v_0v_1v_3-v_0v_2v_3)&=&e^{tf}\partial e^{-tf}(v_0v_1v_3-v_0v_2v_3)\nonumber\\
&=&e^{tf}\partial e^{-tf}(v_0v_1v_3)-e^{tf}\partial e^{-tf}(v_0v_2v_3)\nonumber\\
&=&e^{-tf(v_0v_1v_3)}e^{tf}\partial(v_0v_1v_3)-e^{-tf(v_0v_2v_3)}e^{tf}\partial(v_0v_2v_3) \nonumber\\
&=&[e^{t[f(v_0v_1)-f(v_0v_1v_3)]}v_0v_1+e^{t[f(v_1v_3)-f(v_0v_1v_3)]}v_1v_3]\nonumber\\
&&-[e^{t[f(v_0v_2)-f(v_0v_2v_3)]}v_0v_2+e^{t[f(v_2v_3)-f(v_0v_2v_3)]}v_2v_3]\nonumber\\
&&+[e^{t[f(v_0v_3)-f(v_0v_2v_3)]}v_0v_3-e^{t[f(v_0v_3)-f(v_0v_1v_3)]}v_0v_3]\label{eq-a}.
\end{eqnarray}
Since the coefficient of $v_0v_3$ in (\ref{eq-a}) may  {\bf not} be zero, it follows that
\begin{eqnarray*}
\partial_t(v_0v_1v_3-v_0v_2v_3)\not\in\Omega_1(G).
\end{eqnarray*}

Moreover, 
\begin{eqnarray*}
\partial^{*}\mid_{\Omega(G)}(v_0v_1)&=&v_0v_1v_3,\\
(\partial\mid_{\Omega(G)})^{*}(v_0v_1)&=&v_0v_1v_3-v_0v_2v_3.
\end{eqnarray*}
Hence,
\begin{eqnarray*}
\partial^{*}\mid_{\Omega(G)}&\not=&(\partial\mid_{\Omega(G)})^{*}.
\end{eqnarray*}
Therefore, we will further consider the path homology of general digraphs based on the results of \cite{lin} instead of techniques from Hodge theory.
\end{example}

\section{Description of Path Homology of Digraphs}\label{se-hh}
 In this section, we will characterize the path homology of digraphs by chain complex related to critical sets of transitive closure of digraphs.

\bigskip

\subsection{$\overline{\Phi}$-invariant Module of Transitive Digraph}\label{sec-3}
Firstly, we  give some properties of critical paths on  transitive digraphs. 
\begin{proposition}
Let $G$ be a digraph. Then $G$ is transitive if and only if for any allowed elementary paths $\gamma^{(n+2)}> \alpha^{(n+1)} > \beta^{(n)}$, there exists an allowed elementary $n$-path $\alpha'^{(n+1)}\neq \alpha^{(n+1)}$ such that $\gamma > {\alpha'}  > \beta $.
\end{proposition}
\begin{proof}
Suppose $\gamma^{(n+2)}> \alpha^{(n+1)} > \beta^{(n)}$. Then by \cite[Proposition~2.6]{wang}, there are two cases.

 {\sc Case~1}. There exists an allowed elementary $(n+1)$-path $\alpha'^{(n+1)}\neq \alpha^{(n+1)}$ such that $\gamma > {\alpha'}  > \beta $. 

{\sc Case~2}. $\beta$ is obtained by removing two subsequent vertices $v_i\to v_{i+1}$ in $\gamma$ where $0\leq i \leq n+1$.

Hence the critical part is to verify that the assertion is followed for Case~2.
Let
\begin{eqnarray*}
\gamma=v_0v_1\cdots  v_{n+2}
\end{eqnarray*}
and
\begin{eqnarray*}
\beta=v_0\cdots v_{i-1}v_{i+2}\cdots v_{n+2}.
\end{eqnarray*}
Suppose $G$ is  transitive. Then
\begin{eqnarray*}
v_0\cdots v_{i-1}v_{i+1}v_{i+2}\cdots v_{n+2}
\end{eqnarray*}
and
\begin{eqnarray*}
v_0\cdots v_{i-1}v_i v_{i+2}\cdots v_{n+2}
\end{eqnarray*}
are both allowed elementary paths on $G$ which can be denoted as $\alpha$ and $\alpha'$ respectively. Since $v_i\not=v_{i+1}$, it follows that $\alpha\neq\alpha'$.

Hence, summarizing Case~1 and Case~2, we have that if $G$ is transitive, then there exists an allowed elementary $n$-path $\alpha'^{(n+1)}\neq \alpha^{(n+1)}$ such that $\gamma > {\alpha'}  > \beta $.

On the other hand, suppose for any allowed elementary paths $\gamma^{(n+2)}> \alpha^{(n+1)} > \beta^{(n)}$, there exists an allowed elementary $n$-path $\alpha'^{(n+1)}\neq \alpha^{(n+1)}$ such that $\gamma > {\alpha'}  > \beta $. Let $u\to v$ and $v\to w$ be direct edges of $G$. Let $\gamma=uvw$, $\alpha=uv$ and $\beta=u$.  Then $\alpha'=uw$ must be  an allowed elementary path on $G$. Hence, $G$ is  transitive.
\end{proof}

\begin{lemma}\label{le-1.15}
Let $\alpha=v_0\cdots v_n$ ($n>1$) be a critical path on transitive digraph $G$. Let $f$ be a discrete Morse function on $G$. Then there exists at most one $d_j\alpha$ ($0\leq j\leq n$) such that $d_j\alpha$ is non-critical.
\end{lemma}
\begin{proof}
Since $G$ is  transitive and  $\alpha$ is  critical, it follows that $d_i\alpha$ is allowed on $G$ and $f(v_i)>0$ for any $0\leq i\leq n$. Suppose $\beta=d_j\alpha$ is non-critical for some $0\leq j\leq n$. Then there exists an unique vertex $u\in V(G)$ with $f(u)=0$ such that
\begin{eqnarray*}
\alpha'=v_0\cdots v_{j-1} \hat{v}_j v_{j+1}\cdots v_{k} u v_{k+1}\cdots v_n
\end{eqnarray*}
is an allowed elementary $n$-path on $G$, $\alpha'>\beta$ and $f(\alpha')=f(\beta)$.

{\bf Step~1}. We assert that
\begin{eqnarray}\label{eq-1a}
\alpha'=v_0\cdots v_{j-1}u v_{j+1}\cdots v_n.
\end{eqnarray}
Suppose to the contrary, either
\begin{eqnarray*}
\alpha'=v_0\cdots v_k u v_{k+1}\cdots v_{j-1}v_{j+1}\cdots v_n
\end{eqnarray*}
or
\begin{eqnarray*}
\alpha'=v_0\cdots v_{j-1}v_{j+1}\cdots v_k u v_{k+1}\cdots v_n.
\end{eqnarray*}
Without loss of generality,
\begin{eqnarray*}
\alpha'=v_0\cdots v_{j-1}v_{j+1}\cdots v_k u v_{k+1}\cdots v_n.
\end{eqnarray*}
Then
\begin{eqnarray*}
\gamma=v_0\cdots v_j\cdots v_{k} u v_{k+1}\cdots v_n.
\end{eqnarray*}
is an allowed elementary $(n+1)$-path on $G$ such that  $ \gamma>\alpha$ and $f(\gamma)=f(\alpha)$. This  contradicts that $\alpha$ is  critical. Hence the assertion is proved.

 {\bf Step~2}. We will prove that for any $0\leq i\not=j\leq n$, $d_i\alpha$ is  critical. Suppose to the contrary, there exists an allowed elementary path $\beta'=d_i\alpha$ ($i\not=j$) which is non-critical. Then by a discussion similar to the above, there exists an unique vertex $w\in V(G)$ with $f(w)=0$ and
 \begin{eqnarray}\label{eq-1b}
\alpha''= v_0\cdots v_{i-1}w v_{i+1}\cdots v_n
 \end{eqnarray}
 is an allowed elementary $n$-path on $G$ such that $\alpha''>\beta'$ and $f(\alpha'')=f(\beta')$.
 Without of loss of generality, $0\leq i<j\leq n$.

 Firstly, we assert that $u\not=w$. Suppose to the contrary, $u=w$. Consider the following two cases.

  {\sc Case~1}. $j=i+1$. Then $i=j-1$. By (\ref{eq-1a}) and (\ref{eq-1b}),  $\gamma=v_0\cdots v_{i-1}v_{i}wv_{i+1}\cdots v_n$ is an allowed elementary path on $G$ such that   $\gamma>\alpha$ and $f(\gamma)=f(\alpha)$. This  contradicts that $\alpha$ is critical.

 {\sc Case~2}. $j>i+1$. Then $j-1>i$. Hence, $wv_{i+1}\cdots v_{j-1}w$ (or $uv_{i+1}\cdots v_{j-1}u)$ is a directed loop with $f(w)=0$ (or $f(u)=0$) which contradicts Lemma~\ref{le-1.14}.

 Combining Case~1 and Case~2, $u\not=w$.

Secondly, according to the value of $i$, we divide it into two cases to complete the proof of Step~2.

{\sc Case~3}.  $i\geq 1$. Since $G$ is transitive and $j>i$, $v_{i-1}\to v_{j-1}$ is a directed edge of $G$. Thus, $v_{i-1}\to u$ is also a directed edge of $G$. Moreover, since $u\not=w$, it follows that $v_{i-1}\to u$ and $v_{i-1}\to w$ are two distinct directed edges of $G$ with $f(u)=f(w)=0$. This contradicts that $f$ is a discrete Morse function on $G$.

{\sc Case~4}.  $i=0$. There are two subcases.

{\sc Subcase~4.1}. $j<n$. Since $G$ is transitive and $j>i$, $v_{i+1}\to v_{j+1}$ is a directed edge of $G$. Thus, $w\to v_{j+1}$ is also a directed edge of $G$. Moreover, since $u\not=w$, it follows that $u\to v_{j+1}$ and $w\to v_{j+1}$ are two distinct directed edges of  $G$ with $f(u)=f(w)=0$. This contradicts that $f$ is a discrete Morse function on $G$.

{\sc Subcase~4.2}. $j=n$. Then $wv_{1}\cdots v_{n-1} u$ is an allowed path on $G$  with $f(u)=f(w)=0$. Since $u\not=w$, there are two distinct zero-points in the path $wv_{i+1}\cdots v_{n-1} u$ which contradicts Lemma~\ref{le-1.13}.

By Step~1 and Step~2, the lemma follows.

\end{proof}

\begin{remark}\label{re-1}
Note that in the proof of Lemma~\ref{le-1.15}, the condition $n>1$ ensures that $w\to v_{i+1}$ and $v_{j-1}\to u$ are directed edges of $G$.
If $n=1$,  then Lemma~\ref{le-1.15} may not hold. For example, let $G$ is a digraph with $V(G)=\{v_0,v_1,v_2,v_3\}$ and
\begin{eqnarray*}
E(G)=\{v_0\to v_1, v_2\to v_1, v_0\to v_3\}.
\end{eqnarray*}
Then $G$ is a  transitive digraph. Let $f$ be a function on $G$ with $f(v_2)=f(v_3)=0$ and $f(v_0)>0, f(v_1)>0$.  It is easy to verify that $f$ is  a discrete Morse function and $\alpha=v_0v_1$ is critical. Let $\beta_1=v_0<\alpha$ and $\beta_2=v_1<\alpha$. Then $v_0v_3>\beta_1$ and $f(v_0v_3)=f(\beta_1)$, and  $v_2v_1>\beta_2$ and $f(v_2v_1)=f(\beta_2)$. Hence, $\beta_1$ and $\beta_2$ are both non-critical.
\end{remark}

\smallskip

 Denote $P_{*}^{\overline{\Phi}}(G)$ as the sub-chain complex of $P_*(G)$ consisting of all $\overline{\Phi}$-invariant chains where $G$ is a transitive digraph.

\begin{lemma}\label{le-26}
Let $G$ be a transitive digraph and $f$ a discrete Morse function on $G$. If $\alpha=uv$ is critical and either $\beta_0=u$ or $\beta_1=v$ is not critical. Then $\alpha\notin P_{1}^{\overline{\Phi}}(G)$ and $\big(\alpha+\overline{V}\partial(\alpha)\big)\in P_{1}^{\overline{\Phi}}(G)$.
\end{lemma}
\begin{proof}
Since $\alpha$ is critical, $f(u)>0$ and $f(v)>0$. We divide the proof into the following cases.

{\sc Case~1}. Only one of $\beta_0$ and $\beta_1$ is not critical. Without loss of generality, $\beta_0$ is not critical and $\beta_1$ is critical. Then there exists an unique vertex $w\in V(G)$ such that $f(w)=0$ and  f($uw)=f(u)$. Hence
\begin{eqnarray*}
\alpha+\overline{V}\partial(\alpha)&=&uv+\overline{V}(v-u)\\
&=&uv-uw
\end{eqnarray*}
and
\begin{eqnarray*}
\overline{\Phi}(\alpha+\overline{V}\partial(\alpha))&=&(\mathrm{Id}+\partial\overline{V}+\overline{V}\partial)(uv-uw)\\
&=&(\mathrm{Id}+\partial\overline{V}+\overline{V}\partial)(uv)-(\mathrm{Id}+\partial\overline{V}+\overline{V}\partial)(uw)\\
&=&\big(uv+(\overline{V}\partial)(uv)\big)-\big(uw+(\overline{V}\partial)(uw)\big)\\
&=&\big(uv+\overline{V}(v-u)\big)-\big(uw+\overline{V}(w-u)\big)\\
&=&(uv-uw)\in P_{1}^{\overline{\Phi}}(G).
\end{eqnarray*}

{\sc Case~2}. Both $\beta_0=u$ and $\beta_1=v$ are not critical. Then there exist $\alpha_0>\beta_0$ and $\alpha_1>\beta_1$ such that
$f(\alpha_0)=f(\beta_0)$ and $f(\alpha_1)=f(\beta_1)$ where $\alpha_0$ and $\alpha_1$ are  allowed elementary paths on $G$. By Remark~\ref{re-1}, we assert that the direction of $\alpha_0$ and $\alpha_1$ are not consistent with $\alpha$. Suppose to the contrary, at least one of $\alpha_0$ and $\alpha_1$ is consistent with $\alpha$. Without loss of generality, $\alpha_0=uw$, $\alpha_1=vw'$ and $f(w)=f(w')=0$. Let $\gamma=uvw'$. Then $\gamma>\alpha$ and $f(\gamma)=f(\alpha)$ which contradicts $\alpha$ is critical. Hence, $\alpha_0$ and $\alpha_1$ can be written as $uw$ and $w'v$ respectively with $f(w)=f(w')=0$.

Moreover, we can prove that $w\not=w'$. Suppose to the contrary, $w=w'$. Let $\gamma=uwv$. Then $\alpha<\gamma$ and  $f(\alpha)=f(\gamma)$ which contradicts that $\alpha$ is critical.

Therefore,
\begin{eqnarray*}
\overline{\Phi}(\alpha)&=&(\mathrm{Id}+\partial\overline{V}+\overline{V}\partial)(\alpha)\\
&=&\alpha+\overline{V}\partial(\alpha)\\
&=&uv+\overline{V}(v-u)\\
&=&uv-w'v-uw
\end{eqnarray*}
and
\begin{eqnarray*}
\overline{\Phi}(\alpha+\overline{V}\partial(\alpha))&=&\overline{\Phi}(uv-w'v-uw)\\
&=&(\mathrm{Id}+\partial\overline{V}+\overline{V}\partial)(uv-w'v-uw)\\
&=&(uv-w'v-uw)+\overline{V}\partial(uv-w'v-uw)\\
&=&(uv-w'v-uw)+\overline{V}(w'-w)\\
&=&uv-w'v-uw\\
&=&\alpha+\overline{V}\partial(\alpha).
\end{eqnarray*}
The lemma is proved.
\end{proof}
\bigskip
\begin{lemma}\label{le-l1}
Let $G$ be a transitive digraph and $f$ a discrete Morse function on $G$. Then
\begin{eqnarray*}
\overline{\Phi}(\alpha)=0
\end{eqnarray*}
for any $\alpha\in P(G)$ where $\alpha$ is not critical.
\end{lemma}
\begin{proof}
By Lemma~\ref{le-1.24},  there are two cases.

{\sc Case~1}. 
There exists an unique allowed elementary path $\beta'$ on $G$ such that $\beta'<\alpha$ and $f(\beta')=f(\alpha)$. Then
\begin{eqnarray*}
\overline{V}(\beta')=-\langle\partial \alpha, \beta'\rangle \alpha
\end{eqnarray*}
and by Lemma~\ref{le-1.14}, $\overline{V}(\alpha)=0$.

Hence,
\begin{eqnarray*}
\overline{V}\partial(\alpha)&=&\overline{V}\Big(\sum\limits_{\beta<\alpha}\langle \partial\alpha, \beta\rangle \beta\Big)\\
&=&\langle\partial\alpha,\beta'\rangle\overline{V}(\beta')\\
&=&-{\langle\partial\alpha,\beta'\rangle}^2\alpha\\
&=&-\alpha
\end{eqnarray*}
and
\begin{eqnarray*}
\overline{\Phi}(\alpha)&=&\alpha-\alpha\\
&=&0.
\end{eqnarray*}

{\sc Case~2}. There exists an unique allowed elementary path $\gamma$ on $G$ such that $\gamma>\alpha$ and $f(\gamma)=f(\alpha)$. Then $\overline{V}(\alpha)=-\langle\partial\gamma,\alpha\rangle\gamma\not=0$. 

Let
\begin{eqnarray*}
\alpha=v_0\cdots v_n,\quad \gamma=v_0\cdots v_j u v_{j+1}\cdots v_n
\end{eqnarray*}
where $f(u)=0$. Then
\begin{eqnarray}
(\overline{V}\partial+\partial\overline{V})(\alpha)&=&\sum\limits_{i=0}^{n}\overline{V}\big((-1)^{i}d_i\alpha\big)-\langle\partial\gamma,\alpha\rangle \partial\gamma\nonumber\\
&=&\sum\limits_{i=0}^{n}(-1)^{i}\overline{V}\big(d_i\alpha\big)-(-1)^{j+1} \partial\gamma\label{eq-z4}.
\end{eqnarray}
Consider the following subcases.

{\sc Subcase~2.1}. $0\leq i\leq j$. Then
\begin{eqnarray*}
\overline{V}(d_i\alpha)&=&-\langle\partial\alpha',d_i\alpha\rangle\alpha'\\
&=&-(-1)^{j}\alpha'
\end{eqnarray*}
where
\begin{eqnarray*}
\alpha'=v_0\cdots\hat{v}_i\cdots v_j u v_{j+1}\cdots v_n.
\end{eqnarray*}
Hence, the term containing  $\alpha'$ in (\ref{eq-z4}) is
\begin{eqnarray*}
\overline{V}\big((-1)^{i}d_i(\alpha)\big)-(-1)^{j+1}\langle\partial\gamma,\alpha'\rangle\alpha'&=&-(-1)^{i+j}\alpha'-(-1)^{j+1}(-1)^{i}\alpha'\\
&=&(-1)^{i+j+1}\alpha'-(-1)^{i+j+1}\alpha'\\
&=&0.
\end{eqnarray*}

{\sc Subcase~2.2}. ${j+1}\leq i\leq n$. Then
\begin{eqnarray*}
\overline{V}(d_i\alpha)&=&-\langle\partial\alpha'',d_i\alpha\rangle\alpha''\\
&=&-(-1)^{j+1}\alpha''
\end{eqnarray*}
where
\begin{eqnarray*}
\alpha''=v_0\cdots v_j u v_{j+1}\cdots \hat{v}_i\cdots v_n.
\end{eqnarray*}
Hence, the term containing  $\alpha''$ in (\ref{eq-z4}) is
\begin{eqnarray*}
\overline{V}((-1)^{i}d_i(\alpha))-(-1)^{j+1}\langle\partial\gamma,\alpha''\rangle\alpha''&=&-(-1)^{i}(-1)^{j+1}\alpha''-(-1)^{j+1}(-1)^{i+1}\alpha''\\
&=&(-1)^{i+j+2}\alpha''-(-1)^{i+j+2}\alpha''\\
&=&0.
\end{eqnarray*}

Combing Subcase 2.1 and Subcase 2.2, we have that all terms of $\overline{V}\partial (\alpha)$ are cancelled out with terms of $\partial \overline{V}(\alpha)$ and there is only one item left in $\overline{V}\partial (\alpha)+\partial \overline{V}(\alpha)$. Specifically,
\begin{eqnarray*}
\overline{\Phi}(\alpha)&=&\alpha+\overline{V}\partial (\alpha)+\partial \overline{V}(\alpha)\\
&=&\alpha+(-(-1)^{j+1}(-1)^{j+1})\alpha\\
&=&0.
\end{eqnarray*}
Therefore, the assertion follows.
\end{proof}
\smallskip
Next, we give the characterization of the $\overline{\Phi}$-invariant set of transitive digraphs.

\begin{proposition}\label{prop-11}
Let $G$ be a transitive digraph and $f$ a discrete Morse function on $G$. Then
\begin{eqnarray*}\label{eq-17.2}
P_*^{\overline{\Phi}}({G})=R\big(\alpha+\overline{V}\partial(\alpha)\big)
\end{eqnarray*}
where $\alpha$ is critical in $G$. 
\end{proposition}
\begin{proof}
We divide the proof into the following two steps.

{\bf Step~1}.  We prove that $R\big(\alpha+\overline{V}\partial(\alpha)\big)\subseteq P_*^{\overline{\Phi}}({G})$.

Let $\alpha$ is a critical $n$-path on $G$. Then $\overline{V}(\alpha)=0$ and
\begin{eqnarray*}\label{eq-24}
\overline{\Phi}(\alpha)=\alpha+\overline{V} \partial(\alpha).
\end{eqnarray*}
Consider the following cases according to the value of $\overline{V} \partial(\alpha)$.

{\sc Case~1}. $\beta$ is critical for any $\beta<\alpha$. Then $\overline{V} \partial(\alpha)=0$. Hence,
\begin{eqnarray*}
\overline{\Phi}\big(\alpha+\overline{V}\partial(\alpha)\big)&=&\overline{\Phi}(\alpha)\\
&=&\alpha+\overline{V}\partial(\alpha)
\end{eqnarray*}
which implies that $\big(\alpha+\overline{V}\partial(\alpha)\big)\in P_*^{\overline{\Phi}}({G})$.

{\sc Case~2}. There exists an allowed elementary path  $\beta'<\alpha$ such that $\beta'$ is not critical. Then $n\geq 1$.

Suppose
\begin{eqnarray*}
\alpha=v_0v_1\cdots v_n
\end{eqnarray*}
and $\beta'=d_i\alpha$. There are two subcases.

{\sc Subcase~2.1}. $n>1$. Then by Lemma~\ref{le-1.15},  $\beta'$ is unique and  $\overline{V}(\beta')=-\langle \partial \alpha',\beta'\rangle \alpha'$ where
\begin{eqnarray*}
\alpha'=v_0\cdots v_{i-1}v'_i v_{i+1}\cdots v_n,\quad v'_i\not=v_i,\quad f(v'_i)=0.
\end{eqnarray*}

Notice that,
\begin{eqnarray*}
\langle \partial \alpha,\beta'\rangle=\langle \partial \alpha',\beta'\rangle.
\end{eqnarray*}
Hence,
\begin{eqnarray*}
\overline{V}\partial(\alpha')&=&\overline{V}(\langle \partial \alpha',\beta'\rangle \beta')\\
&=&-{\langle \partial \alpha',\beta'\rangle}^2 \alpha'\\
&=&-\alpha'
\end{eqnarray*}
and
\begin{eqnarray*}
\overline{\Phi}(\alpha+\overline{V} \partial(\alpha))&=&\overline{\Phi}(\alpha+\langle \partial \alpha,\beta'\rangle \overline{V}(\beta'))\\
&=&\overline{\Phi}(\alpha-\alpha')\\
&=&\alpha-\alpha'+\overline{V}\partial(\alpha-\alpha')\\
&=&\alpha-\alpha'\\
&=&\alpha+\overline{V} \partial(\alpha).
\end{eqnarray*}

{\sc Subcase~2.2}. $n=1$. Then  by  Lemma~\ref{le-26}, we have that
\begin{eqnarray*}
\overline{\Phi}(\alpha+\overline{V}\partial(\alpha))=\alpha+\overline{V}\partial(\alpha).
\end{eqnarray*}

Combining Case~1 and Case~2,  we have that
 \begin{eqnarray*}
 \big(\alpha+\overline{V}\partial(\alpha)\big)\in P_*^{\overline{\Phi}}({G})
 \end{eqnarray*}
for any $\alpha\in \mathrm{Crit}(G)$.

{\bf Step~2}. We prove that $P_*^{\overline{\Phi}}({G})\subseteq R\big(\alpha+\overline{V}\partial(\alpha)\big)$.

Let $x=\sum\limits_{i=1}^{m}a_i\alpha_i\in P_*^{\overline{\Phi}}({G})$ where  $a_i\not=0$ for $1\leq i\leq m$ and $\alpha_1,\cdots,\alpha_m$ are distinct allowed elementary $n$-paths on $G$. Consider the following cases.

 {\sc Case~3}. Each $\alpha_i, 1\leq i\leq m$ is not critical. That is, $x$ is a formal linear combination of non-critical paths on $G$. Then by Lemma~\ref{le-l1}, $\overline{\Phi}(x)=0$ which contradicts $x\in P_*^{\overline{\Phi}}({G})$. Hence, this case does not hold.

 {\sc Case~4}. Each $\alpha_i, 1\leq i\leq m$ is critical. Consider the following  two subcases.

 {\sc Subcase~4.1}.  $\overline{V}\partial(\alpha_i)=0$ for any $\alpha_i$. Then 
 \begin{eqnarray*}
 x=\sum\limits_{i=1}^{m}a_i\alpha_i=\sum\limits_{i=1}^{m}a_i\big(\alpha_i+\overline{V}\partial(\alpha_i)\big)
\end{eqnarray*}
and by Case~1 of Step~1, we have that $\overline{\Phi}(x)=x$.

 {\sc Subcase~4.2}. There exists some $\alpha_i$ such that $\overline{V}\partial(\alpha_i)\not=0$. Since $\overline{V}\partial(\alpha_i)$ can not be cancelled out by any critical path and $\overline{\Phi}(x)=x$, it follows that there  must exist some $\alpha_j, 1\leq j\not=\leq m$  such that $\overline{V}\partial(\alpha_j)\not=0$ and  $\overline{V}\partial(\alpha_i)$ is cancelled out by  $\overline{V}\partial(\alpha_j)$. Specifically,
\begin{itemize}
\item
n>1. By Lemma~\ref{le-1.15}, there exists an unique allowed elementary path $\beta$ on $G$ such that $\beta<\alpha_i$ and $\overline{V}(\beta)\not=0$. Then
\begin{eqnarray*}
\overline{V}\partial(\alpha_i)&=&\overline{V}(\langle\partial\alpha_i,\beta\rangle\beta)\\
&=&-\langle\partial\alpha'_i,\beta\rangle\langle\partial\alpha'_i,\beta\rangle\alpha'_i\\
&=&-\alpha'_i
\end{eqnarray*}
where $\alpha'_i>\beta$ and $f(\alpha'_i)=f(\beta)$.

Similarly,
 \begin{eqnarray*}
 \overline{V}\partial(\alpha_j)&=&\overline{V}(\langle\partial\alpha_j,\beta\rangle\beta')\\
&=&-\langle\partial\alpha_j,\beta'\rangle\langle\partial\alpha'_j,\beta'\rangle\alpha'_j\\
&=&-\alpha'_j
 \end{eqnarray*}
 where $\beta'$ is the unique allowed elementary path on $G$ such that $\beta'<\alpha_j$ and $\overline{V}(\beta')\not=0$, and  $\alpha'_j>\beta'$ and $f(\alpha'_j)=f(\beta')$.

Since $\overline{\Phi}(x)=x$, it follows that $a_i\overline{V}\partial(\alpha_i)+a_j\overline{V}\partial(\alpha_j)=0$. Hence, $\alpha'_i=\alpha'_j$ and $a_i=-a_j$ (In fact, $\beta=\beta'$). Thus,
\begin{eqnarray}
a_i(\alpha_i+\overline{V}\partial(\alpha_i))+a_j(\alpha_j+\overline{V}\partial(\alpha_j))&=&a_i\alpha_i-a_i\alpha'_i+a_j\alpha_j-a_j\alpha'_j\nonumber\\
&=&a_i\alpha_i+a_j\alpha_j\label{eq-l2}.
\end{eqnarray}

Therefore, by Subcase~4.1 of Step~2 and (\ref{eq-l2}),
\begin{eqnarray*}
x&=&\overline{\Phi}(x)\nonumber\\
&=&\overline{\Phi}(\sum\limits_{i=1}^{m}a_i\alpha_i)\nonumber\\
&=&\overline{\Phi}\Big(\sum\limits_{i=1}^{m}a_i\big(\alpha_i+\overline{V}\partial(\alpha_i)\big)\Big).\label{eq-l3}
\end{eqnarray*}
By Step~1,
\begin{eqnarray*}
\overline{\Phi}\Big(\sum\limits_{i=1}^{m}a_i\big(\alpha_i+\overline{V}\partial(\alpha_i)\big)\Big)&=&\sum\limits_{i=1}^{m}a_i\big(\alpha_i+\overline{V}\partial(\alpha_i)\big)
\end{eqnarray*}
which implies that
$x$ can be written as a formal linear combination of $\alpha+\overline{V}\partial(\alpha)$ where $\alpha\in \mathrm{Crit}(G)$.
\item
n=1. Suppose $\alpha_i=v_0v_1$ and $\beta=v_0$ is not critical. By Lemma~\ref{le-26}, there exists an unique allowed elementary path $\alpha'_i=v_0v$ such that $\alpha'_i>\beta$ and $f(\alpha'_i)=f(\beta)$. That is,
\begin{eqnarray*}
\overline{V}(\beta)=v_0v.
\end{eqnarray*}
Since $\overline{\Phi}(x)=x$, it follows that $\overline{V}\partial(\alpha_i)$ must be cancelled out by some $\overline{V}\partial(\alpha_j)$, $1\leq j\not=i\leq m$. Let $\alpha_j=u_0u_1$.  Then $\beta'=u_0$ is not critical and
\begin{eqnarray*}
\overline{V}(\beta')=u_0u
\end{eqnarray*}
where $\alpha'_j=u_0u>\beta'$ and $f(\alpha'_j)=f(\beta')$.  Moreover,  $v_0=u_0$, $v=u$ and $a_i=-a_j$. Similarly, if $\beta''=u_1$ is not critical, then there must exist another critical path $\alpha_k$ ($k\not=i,j$) in $x$ such that $\beta''$ is cancelled out by $\overline{V}\partial(\alpha_k)$. Hence,
by finite steps, we can have that
\begin{eqnarray*}
\sum\limits_{i=1}^{m}a_i\overline{V}\partial(\alpha_i)=0.
\end{eqnarray*}
Therefore,
\begin{eqnarray*}
x&=&\sum\limits_{i=1}^{m}a_i\alpha_i\\
&=&\sum\limits_{i=1}^{m}a_i\alpha_i+\sum\limits_{i=1}^{m}a_i\overline{V}\partial(\alpha_i)\\
&=&\sum\limits_{i=1}^{m}a_i\big(\alpha_i+\overline{V}\partial(\alpha_i)\big).
\end{eqnarray*}
\end{itemize}

 {\sc Case~5}. Some $\alpha_i$ are critical and  some are not critical. Without loss of generality, $\alpha_1$ is not critical. Then $\overline{\Phi}(\alpha_1)=0$. Since all $\alpha_i$ are distinct, $\alpha_1$ can not be cancelled out by any critical path and $x\in P_*^{\overline{\Phi}}(G)$, it follows that there must exist some critical path $\alpha_i$ ($ 2\leq i\leq m$) such that $\overline{V}\partial(\alpha_i)=\alpha_1$ and $a_1=a_i$. Thus,
 \begin{eqnarray}\label{eq-l10}
 a_1\alpha_1+a_i\alpha_i=a_i(\alpha_i+\overline{V}\partial(\alpha_i)).
 \end{eqnarray}
Moreover, by (\ref{eq-l10}) and Step~1,
\begin{eqnarray*}
\overline{\Phi}(a_1\alpha_1+a_i\alpha_i)&=&a_i\overline{\Phi}(\alpha_i+\overline{V}\partial(\alpha_i))\\
&=&a_i(\alpha_i+\overline{V}\partial(\alpha_i))\\
&=& a_1\alpha_1+a_i\alpha_i.
\end{eqnarray*}
Hence, any non critical path $\alpha_j$ in $x$ can be written as $\overline{V}\partial(\alpha_i)$ where $\alpha_i$ is critical ($1\leq i\not=j\leq n$). And after removing such non-critical paths $\alpha_j$ and the corresponding critical paths $\alpha_i$ from $x$, the rest is a formal linear combination of critical paths which is invariant under $\overline{\Phi}$, denoted as $x'$. Therefore, by Case~4 of Step~2, $x'$ can be written as  a formal linear combination of $\alpha+\overline{V}\partial(\alpha)$ where $\alpha\in \mathrm{Crit}(G)$. So is $x$.

Combining Case~3,  Case~4 and Case~5,   we have that for any $x\in P_*^{\overline{\Phi}}({G})$, it can be written as a linear combination of $\alpha+\overline{V}\partial(\alpha)$ where $\alpha\in \mathrm{Crit}(G)$.

Therefore, by Step~1 and Step~2, the proposition is proved.
\end{proof}

Finally,  we give a description of path homology groups of digraphs with critical paths of transitive closure.
\begin{theorem}(cf. \cite[Corollary~2.16]{lin})\label{th-111}
Let $G$ be a digraph and $f$ a discrete Morse function on $G$. Let $\bar{f}$ be the extension of $f$ on $\bar{G}$ and $\overline{V}=\mathrm{grad}\bar{f}$ the discrete gradient vector field on $\bar{G}$. Suppose $\Omega_*(G)$ is $\overline{V}$-invariant ($\overline{V}(\Omega_*(G))\subseteq \Omega_*(G)$). Then
\begin{eqnarray*}
H_m(G;R)\cong H_m(\Omega_*^{\overline{\Phi}}(G)), m \geq 0
\end{eqnarray*}
where
\begin{eqnarray*}
\Omega_*^{\overline{\Phi}}(G)=\Omega_*(G)\cap P_*^{\overline{\Phi}}(\bar{G}).
\end{eqnarray*}
\end{theorem}
By Proposition~\ref{prop-11} and Theorem~\ref{th-111}, we have that
\begin{theorem}\label{th-999}
Let $G$ be a digraph and $f$ a discrete Morse function on $G$. Let $\bar{f}$ be the extension of $f$ on $\bar{G}$ and $\overline{V}=\mathrm{grad}\bar{f}$ the discrete gradient vector field on $\bar{G}$. Suppose $\Omega_*(G)$ is $\overline{V}$-invariant. Then
\begin{eqnarray*}
H_m(G;R)\cong  H_m\big(\{R(\alpha+\overline{V}\partial(\alpha))\cap \Omega_n(G),\partial_n\}_{n\geq 0}\big).
\end{eqnarray*}
where $\alpha\in \mathrm{Crit}_n(\bar{G})$.
\end{theorem}

\smallskip

We give an example to illustrate Theorem~\ref{th-999}.
\begin{example}(cf. \cite[Example~3.2]{lin})\label{ex-w1}
Let $G$ be a square as follows and $\bar{G}$ the transitive closure of $G$. Let $f$ be a function on $G$ with
\begin{eqnarray*}
f(v_1)=0, f(v_i)>0 \text{ for } i=0,2,3.
\end{eqnarray*}
\begin{figure}[!htbp]
 \begin{center}
\begin{tikzpicture}[line width=1.5pt]

\coordinate [label=left:$v_0$]    (A) at (1,0);
 \coordinate [label=right:$v_1$]   (B) at (2.5,1);
 \coordinate  [label=right:$v_2$]   (C) at (2.5,-1);
\coordinate  [label=right:$v_3$]   (D) at (4,0);
\fill (1,0) circle (2.5pt);
\fill (2.5,1) circle (2.5pt);
\fill (2.5,-1) circle (2.5pt);
\fill (4,0) circle (2.5pt);

\coordinate[label=left:$G$:] (G) at (0.5,0);
\draw [line width=1.5pt]  (A) -- (B);
\draw [line width=1.5pt]  (A) -- (C);
\draw [line width=1.5pt]   (B) -- (D);
\draw [line width=1.5pt]  (C) -- (D);
 \draw[->] (1,0) -- (2,2/3);

\draw[->] (2.5,1) -- (3,2/3);

\draw[->] (1,0) -- (2,-2/3);

\draw[->] (2.5,-1) -- (3,-2/3);

\coordinate [label=left:$v_0$]    (A) at (1+6,0);
 \coordinate [label=right:$v_1$]   (B) at (2.5+6,1);
 \coordinate  [label=right:$v_2$]   (C) at (2.5+6,-1);
\coordinate  [label=right:$v_3$]   (D) at (4+6,0);
\fill (1+6,0) circle (2.5pt);
\fill (2+6.5,1) circle (2.5pt);
\fill (2.5+6,-1) circle (2.5pt);
\fill (4+6,0) circle (2.5pt);

\coordinate[label=left:$\bar{G}$:] (H) at (6+0.5,0);
\draw [line width=1.5pt]  (A) -- (B);
 \draw [line width=1.5pt]  (A) -- (C);
 \draw [line width=1.5pt]  (A) -- (D);
\draw [line width=1.5pt] (B) -- (D);
\draw [line width=1.5pt]  (C) -- (D);
 \draw[->] (7,0) -- (8,2/3);

\draw[->]  (8.5,1) -- (9,2/3);

\draw[->]  (7,0) -- (8,-2/3);
\draw[->] (8.5,-1) -- (9,-2/3);
 \draw[->] (7,0) -- (9,0);

 \end{tikzpicture}
\end{center}

\caption{Example~\ref{ex-w1}.}
\end{figure}
It is easy to verify that $f$ is a discrete Morse function on $G$ and $f$ can be extended to be a discrete Morse function $\bar{f}$ on $\bar{G}$. Then
\begin{eqnarray*}
P_*(\bar{G})&=&\{v_0,v_1,v_2,v_3,v_0v_1,v_0v_2,v_0v_3,v_1v_3,v_2v_3,v_0v_1v_3,v_0v_2v_3\},\\
\Omega_*(G)&=&\{v_0,v_1,v_2,v_3,v_0v_1,v_0v_2,v_1v_3,v_2v_3,v_0v_1v_3-v_0v_2v_3\},\\
\mathrm{Crit}_*(\bar{G})&=&\{v_1,v_2,v_0v_2,v_2v_3,v_0v_2v_3\}.
\end{eqnarray*}
Let $\overline{V}=\mathrm{grad}\bar{f}$ be the discrete gradient vector field on $\bar{G}$ and $\overline \Phi=\mathrm{Id}+ \partial \overline V+\overline V\partial$ the discrete gradient flow of $\bar{G}$. Then
\begin{eqnarray*}
&& \overline{V}(v_0)=v_0v_1, ~~~\overline{V}(v_3)=-v_1v_3,\\
&& \overline{V}(v_0v_3)=v_0v_1v_3,\\
&& \overline{V}(\alpha)=0 ~ \text{for any other allowed elementary path}~\alpha \text{ on}~ \bar{G}.
\end{eqnarray*}
and
\begin{eqnarray*}
&\overline \Phi(v_0)=v_1,&\overline \Phi(v_1)=v_1,\\
&\overline \Phi(v_2)=v_2,&\overline \Phi(v_3)=v_1,\\
&\overline \Phi(v_0v_1)=0,&\overline \Phi(v_0v_2)=v_0v_2-v_0v_1,\\
&\overline \Phi(v_1v_3)=0,&\overline \Phi(v_2v_3)=v_2v_3-v_1v_3,\\
&\overline \Phi(v_0v_3)=0,&\overline \Phi(v_0v_1v_3)=0,\\
&~~~~~~~~~~~~~~~~~~~~\overline \Phi(v_0v_2v_3)=v_0v_2v_3-v_0v_1v_3.
\end{eqnarray*}
Hence,
\begin{eqnarray*}
\overline{V}(\Omega_n(G))\subseteq \Omega_{n+1}(G), \quad n\geq 0.
\end{eqnarray*}
and
\begin{eqnarray*}
R(\{\alpha+\overline{V}\partial(\alpha)\mid \alpha\in \mathrm{Crit}(\bar{G})\})&=&P_*^{\Phi}(\bar{G})\\
&=&R(v_1, v_2, v_0v_2-v_0v_1, v_2v_3-v_1v_3, v_0v_2v_3-v_0v_1v_3).
\end{eqnarray*}

Therefore,
\begin{eqnarray*}
&&\partial(v_0v_2-v_0v_1)=v_2-v_1,~~ \partial(v_2v_3-v_1v_3)=v_1-v_2,\\
&&\partial(v_0v_2v_3-v_0v_1v_3)=(v_0v_2-v_0v_1)+(v_2v_3-v_1v_3),
\end{eqnarray*}
and
\begin{eqnarray}
H_0\big(\{R(\alpha+\overline{V}\partial(\alpha))\cap \Omega_n(G),\partial_n\}_{n\geq 0}\big)&=&R,\label{eq-b1}\\
H_m\big(\{R(\alpha+\overline{V}\partial(\alpha))\cap \Omega_n(G),\partial_n\}_{n\geq 0}\big)&=&0 ~\text{for} ~m\geq 1,\label{eq-b2}
\end{eqnarray}
where $\alpha\in \mathrm{Crit}_n(\bar{G})$. (\ref{eq-b1}) and (\ref{eq-b2}) are consistent with the  path homology groups  $H_m(G;R)$ ($m\geq 0$) given in \cite[Proposition~4.7]{9}.

\end{example}

\subsection{Description of Path Homology of Digraphs}\label{sec-4}
In this section, we prove that path homology of digraph $G$ is  isomorphic to the homology  of  chain complex  consisting of the formal linear combinations of paths  which are not only critical paths of the transitive closure $\bar{G}$ but also allowed elementary paths of $G$.

\bigskip

Firstly, we  prove an  isomorphism of graded $R$-modules  in the following theorem.
\begin{theorem}\label{th-aa}
Let $G$ be a digraph and $\bar{G}$ the transitive closure of $G$. Let $\bar{f}$ be a discrete Morse function on $\bar{G}$. Suppose $\overline{\Phi}(\alpha)\in \Omega(G)$ for any $\alpha\in \mathrm{Crit}(\bar{G})\cap P(G)$ where $\overline{V}$ is the discrete gradient vector field on $\bar{G}$ and $\overline{\Phi}$ is the discrete gradient flow of $\bar{G}$. Then
\begin{eqnarray}\label{eq-cc}
\overline{\Phi}^{\infty}\mid_{\mathrm{Crit}_n(\bar{G})\cap P_n(G)}: \mathrm{Crit}_n(\bar{G})\cap P_n(G)\longrightarrow  P_n^{\overline{\Phi}}(\bar{G})\cap\Omega_n(G), ~~n\geq 0
\end{eqnarray}
is a graded $R$-module isomorphism.
\end{theorem}
\begin{proof}
Let  $\alpha\in \mathrm{Crit}_n(\bar{G})\cap P_n(G)$.  By Step~1 of Proposition~\ref{prop-11}, we have  that
\begin{eqnarray*}
\overline{\Phi}(\alpha)=\big(\alpha+\overline{V}\partial(\alpha)\big)\in P_n^{\overline{\Phi}}(\bar{G}).
\end{eqnarray*}
Then
\begin{eqnarray}\label{eq-l12}
\overline{\Phi}^{\infty}(\alpha)=\big(\alpha+\overline{V}\partial(\alpha)\big)\in P_n^{\overline{\Phi}}(\bar{G}).
\end{eqnarray}

Since
\begin{eqnarray*}
\overline{\Phi}(\alpha)\in\Omega_n(G),
\end{eqnarray*}
it follows that
\begin{eqnarray*}
\overline{\Phi}^{\infty}(\alpha)=\overline{\Phi}(\alpha)\in\Omega_n(G).
\end{eqnarray*}

Hence,
\begin{eqnarray*}
\overline{\Phi}^{\infty}(\alpha)\in  P_n^{\overline{\Phi}}(\bar{G})\cap\Omega_n(G)
\end{eqnarray*}
which implies that (\ref{eq-cc}) is well-defined.

Suppose $\alpha', \alpha''\in \mathrm{Crit}_n(\bar{G})\cap P_n(G)$ where $\alpha', \alpha''$ are distinct. Then by (\ref{eq-l12}),
\begin{eqnarray*}
\overline{\Phi}^{\infty}(\alpha')&=&\alpha'+\overline{V}\partial(\alpha')\label{eq-l6}\\
\overline{\Phi}^{\infty}(\alpha'')&=&\alpha''+\overline{V}\partial(\alpha'')\label{eq-l7}.
\end{eqnarray*}
We assert that
\begin{eqnarray*}
\overline{\Phi}^{\infty}(\alpha')\not=\overline{\Phi}^{\infty}(\alpha'').
\end{eqnarray*}
Suppose to the contrary,
\begin{eqnarray}\label{eq-n1}
\alpha'+\overline{V}\partial(\alpha')=\alpha''+\overline{V}\partial(\alpha'').
\end{eqnarray}
Let $\alpha$ be an arbitrary allowed elementary $n$-path on $\bar{G}$. Then either $\overline{V}(\alpha)=0$  or $\overline{V}(\alpha)=\gamma$ where $\gamma$ is an allowed elementary $(n+1)$-path on $\bar{G}$ such that $\gamma>\alpha$ and $\bar{f}(\gamma)=\bar{f}(\alpha)$. Hence, both $\overline{V}\partial(\alpha')$ and $\overline{V}\partial(\alpha'')$ are either equal to $0$ or formal linear combinations of non-critical paths on $\bar{G}$.
Then by (\ref{eq-n1}), $\alpha=\alpha'$ which contradicts that $\alpha$ and $\alpha'$ are distinct. Thus, (\ref{eq-cc}) is a monomorphism.

Moreover, by Proposition~\ref{prop-11}, we know that
\begin{eqnarray*}
P_n^{\overline{\Phi}}(\bar{G})=R\big(\alpha+\overline{V}\partial(\alpha)\big)
\end{eqnarray*}
where $\alpha\in \mathrm{Crit}_n(\bar{G})$. Let $x=\sum\limits_{i=1}^{m}a_i\big(\alpha_i+\overline{V}\partial(\alpha_i)\big)\in P_n^{\overline{\Phi}}(\bar{G})$ where $a_i\in R$ and $\alpha_i\in \mathrm{Crit}(\bar{G})$. For any $1\leq j\not=k\leq m$, since the pairs
\begin{eqnarray*}
&&\alpha_j~ \mathrm{and}~\alpha_k \quad\quad\quad \alpha_j ~\mathrm{and}~ \overline{V}\partial(\alpha_j)\quad\quad\quad \alpha_j ~\mathrm{and}~ \overline{V}\partial(\alpha_k)\\
&&\alpha_k~ \mathrm{and}~\overline{V}\partial(\alpha_j)\quad\quad \alpha_k~ \mathrm{and}~\overline{V}\partial(\alpha_k)
\end{eqnarray*}
 can not cancel out each other in $x$, it follows that if $x\in P_n(G)$, then $\alpha_i\in P_n(G)$ for each $1\leq i\leq m$. Hence,
\begin{eqnarray*}
P_n^{\overline{\Phi}}(\bar{G})\cap \Omega_n(G)\subseteq \{x=\sum\limits_{i=1}^{m}a_{i}(\alpha_i+\overline{V}\partial(\alpha_i))\mid{\alpha_i\in \mathrm{Crit}_n(\bar{G})\cap P_n(G), \partial x\in P(G)}\}.
\end{eqnarray*}
Furthermore, since
\begin{eqnarray*}
\overline{\Phi}(\alpha)=(\alpha+\overline{V}\partial(\alpha))\in\Omega(G)
\end{eqnarray*}
and by Proposition~\ref{prop-11},
\begin{eqnarray*}
\overline{\Phi}(\alpha)\in P_*^{\overline{\Phi}}(\bar{G})
\end{eqnarray*}
where  $\alpha\in \mathrm{Crit}(\bar{G})\cap P(G)$, it follows that
\begin{eqnarray*}
P_n^{\overline{\Phi}}(\bar{G})\cap \Omega_n(G)\supseteq
\{x=\sum\limits_{i=1}^{m}a_{i}(\alpha_i+\overline{V}\partial(\alpha_i))\mid{\alpha_i\in \mathrm{Crit}_n(\bar{G})\cap P_n(G)}\}.
\end{eqnarray*}
Hence,
\begin{eqnarray*}
P_n^{\overline{\Phi}}(\bar{G})\cap \Omega_n(G)=
\{x=\sum\limits_{i=1}^{m}a_{i}(\alpha_i+\overline{V}\partial(\alpha_i))\mid{\alpha_i\in \mathrm{Crit}_n(\bar{G})\cap P_n(G)}\}.
\end{eqnarray*}

That is, $P_n^{\overline{\Phi}}(\bar{G})\cap\Omega_n(G)$ is a $R$-module of  all the formal linear combinations of elements in the form
\begin{eqnarray*}
\alpha+\overline{V}\partial(\alpha)
\end{eqnarray*}
where $\alpha\in  \mathrm{Crit}_n(\bar{G})\cap P_n(G)$. This implies that (\ref{eq-cc}) is an epimorphism. Therefore, (\ref{eq-cc}) is an isomorphism.
\end{proof}
Next, we give an isomorphism of homology groups.
\begin{corollary}\label{co-aa}
Let $G$ be a digraph and $\bar{G}$ the transitive closure of $G$. Suppose $\Omega_*(G)$ is $\overline{V}$-invariant and $\overline{\Phi}(\alpha)\in \Omega(G)$  for any $\alpha\in \mathrm{Crit}(\bar{G})\cap P(G)$ where $\overline{V}$ is the discrete gradient vector field on $\bar{G}$ and $\overline{\Phi}$ is the discrete gradient flow of $\bar{G}$, respectively. Then
\begin{eqnarray}\label{eq-k1}
H_m(\{\mathrm{Crit}_n(\bar{G})\cap P_n(G),\tilde{\partial}_n\}_{n\geq 0})\cong H_m(G;R)
\end{eqnarray}
where $\tilde{\partial}=(\overline{\Phi}^{\infty})^{-1}\circ\partial \circ \overline{\Phi}^{\infty}$ and $\overline{\Phi}^{\infty}$ is the stabilization map of $\overline{\Phi}$.
\end{corollary}
\begin{proof}
Let $\alpha\in \mathrm{Crit}_n(\bar{G})\cap P_n(G)$. Since $\overline{\Phi}(\alpha)\in \Omega(G)$  for any $\alpha\in \mathrm{Crit}(\bar{G})\cap P(G)$ , by the proof of Theorem~\ref{th-aa}, it follows that
\begin{eqnarray*}
\overline{\Phi}^{\infty}(\alpha)=\alpha+\overline{V}\partial(\alpha)\in P_n^{\overline{\Phi}}(\bar{G})\cap \Omega_n(G).
\end{eqnarray*}
By \cite[Theorem~2.14 (iii)]{lin}, $\{P_{*}^{\overline{\Phi}}(\bar{G})\cap \Omega_{*}(G),\partial_*\}$ is a chain complex. Then
\begin{eqnarray*}
\partial\overline{\Phi}^{\infty}(\alpha)\in P_{n-1}^{\overline{\Phi}}(\bar{G})\cap \Omega_{n-1}(G).
\end{eqnarray*}
 By Theorem~\ref{th-aa}, we know that
 \begin{eqnarray*}
 \overline{\Phi}^{\infty}\mid_{\mathrm{Crit}_*(\bar{G})\cap P_*(G)}
 \end{eqnarray*}
 is an isomorphism. Hence
 \begin{eqnarray*}
\tilde\partial(\alpha)= (\overline{\Phi}^{\infty})^{-1}\circ\partial \circ \overline{\Phi}^{\infty}(\alpha)\in \mathrm{Crit}_{n-1}(\bar{G})\cap P_{n-1}(G).
 \end{eqnarray*}
 Thus, $\{\mathrm{Crit}_n(\bar{G})\cap P_n(G),\tilde\partial_n\}_{n\geq 0}$ is a chain complex.

Moreover, for each $n\geq 0$, $\overline{\Phi}^{\infty}\circ \tilde{\partial}_n=\partial_n\circ\overline{\Phi}^{\infty}$. That is,  the following diagram
\begin{eqnarray*}
 \xymatrix{
\mathrm{Crit}_n(\bar{G})\cap P_n(G)\ar[d]^{\overline{\Phi}^{\infty}}\ar[r]^{\tilde\partial_n}& \mathrm{Crit}_{n-1}(\bar{G})\cap P_{n-1}(G)\ar[d]^{\overline{\Phi}^{\infty}}\\
P_{n}^{\overline{\Phi}}(\bar{G})\cap \Omega_{n}(G)\ar[r]^{\partial_n}&  P_{n-1}^{\overline{\Phi}}(\bar{G})\cap \Omega_{n-1}(G)
 }
 \end{eqnarray*}
is commutative.

Hence, by Theorem~\ref{th-aa},
\begin{eqnarray}\label{eq-l8}
H_m(\{\mathrm{Crit}_n(\bar{G})\cap P_n(G),\tilde{\partial}_n\}_{n\geq 0})\cong H_m(\{P_{n}^{\overline{\Phi}}(\bar{G})\cap \Omega_{n},\partial_n\}_{n\geq0}).
\end{eqnarray}
By Theorem~\ref{th-111},
\begin{eqnarray}\label{eq-l9}
H_m(G;R)\cong H_m(\{P_{n}^{\overline{\Phi}}(\bar{G})\cap \Omega_{n},\partial_n\}_{n\geq0}).
\end{eqnarray}
Therefore, by (\ref{eq-l8}) and (\ref{eq-l9}),
\begin{eqnarray*}
H_m(\{\mathrm{Crit}_n(\bar{G})\cap P_n(G),\tilde{\partial}_n\}_{n\geq 0})\cong H_m(G;R).
\end{eqnarray*}
\end{proof}
\bigskip

Finally, we give some examples to illustrate Corollary~\ref{co-aa}.
\begin{example}(cf. \cite[Example~3.5]{lin})\label{ex-aa}
Consider the following digraph $G$ and its transitive closure $\bar{G}$. Let $f: V(G)\longrightarrow [0, +\infty)$ be  a function on $G$ with $f(v_1)=0$ and $f(v_i)>0$, $0\leq i \leq 5, i\not=1$. It's easy to verify that $f$ can be extended to be a Morse function $\bar{f}$ on $\bar{G}$ such that $\bar{f}(v)=f(v)$ for all vertices $v\in V(G)$. Then
\begin{eqnarray*}
P(G)&=&\{v_0,v_1,v_2,v_3,v_4,v_5,v_0v_1,v_0v_2,v_1v_3,v_1v_4,v_2v_3,v_2v_4,\\
&&v_5v_3,v_5v_4,v_0v_1v_3,v_0v_1v_4,v_0v_2v_3,v_0v_2v_4\}\\
\Omega(G)&=&\{v_0, v_1, v_2, v_3, v_4, v_5,  v_0v_1, v_0v_2, v_1v_3, v_1v_4,  v_2v_3, v_2v_4,\\
&&v_5v_3, v_5v_4, v_0v_1v_3-v_0v_2v_3, v_0v_1v_4-v_0v_2v_4\}\\
\mathrm{Crit}(\bar{G})&=&\{v_1,v_2,v_5,v_0v_2,v_2v_3,v_2v_4,v_5v_3,v_5v_4,v_0v_2v_3,v_0v_2v_4\}\\
\mathrm{Crit}(\bar{G})\cap P(G)&=&\{v_1,v_2,v_5,v_0v_2,v_2v_3,v_2v_4,v_5v_3,v_5v_4,v_0v_2v_3,v_0v_2v_4\}.
\end{eqnarray*}
Let $\overline{V}=\mathrm{grad}\bar{f}$ be the discrete gradient vector field on $\bar{G}$. Then
\begin{eqnarray}
&&\overline{V}(v_0)=v_0v_1, ~~\overline{V}(v_3)=-v_1v_3,~~\overline{V}(v_4)=-v_1v_4,\nonumber\\
&&\overline{V}(v_0v_3)=v_0v_1v_3,~~ \overline{V}(v_0v_4)=v_0v_1v_4,\nonumber\\
&& \overline{V}(\alpha)=0 ~ \text{for any other allowed elementary path}~\alpha \text{ on}~ \bar{G}.\nonumber
\end{eqnarray}
Hence $\Omega(G)$ is $\overline{V}$-invariant.
\begin{figure}[!htbp]
 \begin{center}
\begin{tikzpicture}[line width=1.5pt]

\coordinate [label=left:$v_0$]    (A) at (3,1);
 \coordinate [label=left:$v_1$]   (B) at (2,2);
 \coordinate  [label=right:$v_2$]   (C) at (4,2);
\coordinate  [label=left:$v_3$]   (D) at (2,4);
\coordinate  [label=right:$v_4$]   (E) at (4,4);
 \coordinate [label=left:$v_5$]   (F) at (3,5);

\coordinate[label=left:$G$:] (H) at (1,3);
   \draw [line width=1.5pt]  (A) -- (B);
    \draw [line width=1.5pt]  (A) -- (C);
      \draw [line width=1.5pt]  (B) -- (D);
       \draw [line width=1.5pt]  (B) -- (E);
            \draw [line width=1.5pt]  (C) -- (D);
                 \draw [line width=1.5pt]  (C) -- (E);
                  \draw [line width=1.5pt]  (F) -- (D);
 \draw [line width=1.5pt]  (F) -- (E);
 \draw[->] (3,1) -- (2.5,1.5);
 \draw[->] (3,1) -- (3.5,1.5);

\draw[->] (2,2) -- (2,3.5);
\draw[->] (2,2) -- (3.8,3.8);

\draw[->] (4,2) -- (4,3.5);
\draw[->] (4,2) -- (2.2,3.8);

\draw[->] (3,5) -- (2.5,4.5);
\draw[->] (3,5) -- (3.5,4.5);

\fill (3,1)  circle (2.5pt) (2,2)  circle (2.5pt)  (4,2) circle (2.5pt) (2,4) circle (2.5pt) (4,4) circle (2.5pt) (3,5)  circle (2.5pt);

\coordinate [label=left:$v_0$]    (A) at (3+6,1);
 \coordinate [label=left:$v_1$]   (B) at (2+6,2);
 \coordinate  [label=right:$v_2$]   (C) at (4+6,2);
\coordinate  [label=left:$v_3$]   (D) at (2+6,4);
\coordinate  [label=right:$v_4$]   (E) at (4+6,4);
 \coordinate [label=left:$v_5$]   (F) at (3+6,5);

\coordinate[label=left:$\bar{G}$:] (I) at (1+6,3);
   \draw [line width=1.5pt]  (A) -- (B);
    \draw [line width=1.5pt]  (A) -- (C);
    \draw [line width=1.5pt]  (A) -- (D);
    \draw [line width=1.5pt]  (A) -- (E);
      \draw [line width=1.5pt]  (B) -- (D);
       \draw [line width=1.5pt]  (B) -- (E);
            \draw [line width=1.5pt]  (C) -- (D);
                 \draw [line width=1.5pt]  (C) -- (E);
                  \draw [line width=1.5pt]  (F) -- (D);
 \draw [line width=1.5pt]  (F) -- (E);
 \draw[->] (3+6,1) -- (2.5+6,1.5);
 \draw[->] (3+6,1) -- (3.5+6,1.5);
 \draw[->] (3+6,1) -- (8.4,2.8);
\draw[->] (3+6,1) -- (9.6,2.8);

\draw[->] (2+6,2) -- (2+6,3.5);
\draw[->] (2+6,2) -- (3.8+6,3.8);

\draw[->] (4+6,2) -- (4+6,3.5);
\draw[->] (4+6,2) -- (2+6.2,3.8);

\draw[->] (3+6,5) -- (2.5+6,4.5);
\draw[->] (3+6,5) -- (3.5+6,4.5);

\fill (3+6,1)  circle (2.5pt) (2+6,2)  circle (2.5pt)  (4+6,2) circle (2.5pt) (2+6,4) circle (2.5pt) (4+6,4) circle (2.5pt) (3+6,5)  circle (2.5pt);

 \end{tikzpicture}
\end{center}
\caption{Example~\ref{ex-aa}.}
\end{figure}
Let $\overline \Phi=\mathrm{Id}+ \partial \overline V+\overline V\partial$  be the discrete gradient flow of $\bar{G}$. Then
\begin{eqnarray*}
&\overline \Phi(v_0)=v_1,             &\overline \Phi(v_1)=v_1,\\
&\overline \Phi(v_2)=v_2,            &\overline \Phi(v_3)=v_1,\\
&\overline \Phi(v_4)=v_1,             &\overline \Phi(v_5)=v_5,\\
&\overline \Phi(v_0v_1)=0,            &\overline \Phi(v_0v_2)=v_0v_2-v_0v_1,\\
&\overline \Phi(v_0v_3)=0,            &\overline \Phi(v_0v_4)=0,\\
&\overline \Phi(v_1v_3)=0,            &\overline \Phi(v_1v_4)=0,\\
&\overline \Phi(v_0v_1v_3)=0,         &\overline \Phi(v_2v_4)=v_2v_4-v_1v_4,\\
&\overline \Phi(v_0v_1v_4)=0,         &\overline \Phi(v_5v_4)=v_5v_4-v_1v_4,\\
&\overline \Phi(v_2v_3)=v_2v_3-v_1v_3,&\overline \Phi(v_0v_2v_3)=v_0v_2v_3-v_0v_1v_3,\\
&\overline \Phi(v_5v_3)=v_5v_3-v_1v_3,&\overline \Phi(v_0v_2v_4)=v_0v_2v_4-v_0v_1v_4,\\
\end{eqnarray*}
and for any $\alpha\in \mathrm{Crit}(\bar{G})\cap P(G)$, $\overline{\Phi}(\alpha)\in \Omega(G)$.

By calculate directly, we have that $\overline \Phi^\infty=\overline \Phi$.  Then
\begin{eqnarray*}
\tilde{\partial}(v_0v_2)&=&v_2-v_1,\\
\tilde{\partial}(v_2v_3)&=&v_1-v_2,\\
\tilde{\partial}(v_2v_4)&=&v_1-v_2,\\
\tilde{\partial}(v_5v_3)&=&v_1-v_5,\\
\tilde{\partial}(v_5v_4)&=&v_1-v_5,\\
\tilde{\partial}(v_0v_2v_3)&=&v_0v_2+v_2v_3,\\
\tilde{\partial}(v_0v_2v_4)&=&v_0v_2+v_2v_4.\\
\end{eqnarray*}
Therefore,
\begin{eqnarray*}
H_0(\{\mathrm{Crit}(\bar{G})\cap P(G),\tilde\partial\})&=&R\\
H_1(\{\mathrm{Crit}(\bar{G})\cap P(G),\tilde\partial\})&=&R\\
H_m(\{\mathrm{Crit}(\bar{G})\cap P(G),\tilde\partial\})&=&0~ \text{for}~m\geq 2,
\end{eqnarray*}
which are consistent with the  path homology groups of $G$.
\end{example}
\bigskip
By \cite[Remark~3.6]{lin},  the condition that  $\Omega(G)$ is $\overline{V}$-invariant  in Theorem~\ref{th-111} is sufficient but not necessary. The following example  will show us that the condition
\begin{eqnarray}\label{eq-ee}
\overline{\Phi}(\mathrm{Crit}(\bar{G})\cap P(G))\subseteq  \Omega(G)
\end{eqnarray}
in Corollary~\ref{co-aa} is also sufficient but not necessary.
\begin{example}(cf. \cite[Example~3.5]{lin})\label{ex-bb}
We still consider the digraph $G$ in Example~\ref{ex-aa}. Different from Example~\ref{ex-aa},  we firstly define $f$ with $f(v_0)=0$ and $f(v_i)>0$, $0<i \leq 5$. It is easy to verify that $f$ is a discrete Morse function on $G$ and it can be extended to be a discrete Morse function $\bar{f}$ on $\bar{G}$. By \cite[Example~3.5]{lin}, we have that
\begin{eqnarray*}
\mathrm{Crit}(\bar{G})&=&\{v_0, v_5, v_5v_3, v_5v_4\},\\
\mathrm{Crit}(\bar{G})\cap P(G)&=&\{v_0, v_5, v_5v_3, v_5v_4\},\\
\overline \Phi(v_5v_3)&=&(v_5v_3-v_0v_3)\not\in \Omega(G),\\
\overline \Phi(v_5v_4)&=&(v_5v_4-v_0v_4)\not\in \Omega(G).\\
\end{eqnarray*}
Hence, (\ref{eq-ee}) does not hold.

On the other hand,
\begin{eqnarray*}
\tilde\partial(v_5v_3)&=&v_0-v_5,\\
\tilde\partial(v_5v_4)&=&v_0-v_5.\\
\end{eqnarray*}
Therefore,
\begin{eqnarray*}
H_0(\{\mathrm{Crit}(\bar{G})\cap P(G),\tilde\partial\})&=&R\\
H_1(\{\mathrm{Crit}(\bar{G})\cap P(G),\tilde\partial\})&=&R\\
H_m(\{\mathrm{Crit}(\bar{G})\cap P(G),\tilde\partial\})&=&0~ \text{for}~m\geq 2,
\end{eqnarray*}
which are consistent with the  path homology groups of $G$.
\end{example}

\smallskip
\begin{example}\label{ex-h.17}
Consider the following digraph $G$ and its transitive closure $\bar{G}$. Let $f: V(G)\longrightarrow [0, +\infty)$ be  a function on $G$ with $f(v_0)=0$ and $f(v_i)>0, 0<i\leq 3$. Then $f$ is  a discrete Morse function on $G$ which can be extended to $\bar{G}$. Hence,
\begin{eqnarray*}
\Omega(G)&=&\{v_0, v_1, v_2, v_3, v_0v_1, v_1v_2, v_2v_3, v_0v_3\}\\
\mathrm{Crit}(\bar{G})&=&\{v_0\}\\
\mathrm{Crit}(\bar{G})\cap P(G)&=&\{v_0\}
\end{eqnarray*}
and
\begin{eqnarray*}
\overline{V}(v_2v_3)=-v_0v_2v_3\not\in \Omega(G)\quad \quad \overline{V}(v_1v_2)=-v_0v_1v_2\not\in \Omega(G).
\end{eqnarray*}
However,
\begin{eqnarray*}
H_0(\{\mathrm{Crit}(\bar{G})\cap P(G),\tilde\partial\})&=&R,\\
H_m(\{\mathrm{Crit}(\bar{G})\cap P(G),\tilde\partial\})&=&0~\text{for } m>0
\end{eqnarray*}
which are not consistent with the path homology groups of $G$ (cf. \cite[Proposition~4.7]{9}).
\begin{figure}[!htbp]
 \begin{center}
\begin{tikzpicture}[line width=1.5pt]

\coordinate [label=left:$v_0$]    (A) at (2,2);
 \coordinate [label=left:$v_1$]   (B) at (2,4);
 \coordinate  [label=right:$v_2$]   (C) at (4,4);
\coordinate  [label=right:$v_3$]   (D) at (4,2);
\fill (2,2) circle (2.5pt);
\fill (2,4) circle (2.5pt);
\fill (4,4) circle (2.5pt);
\fill (4,2) circle (2.5pt);

\coordinate[label=left:$G$:] (G) at (1,3);
\draw [line width=1.5pt]  (A) -- (B);
\draw [line width=1.5pt]  (B) -- (C);
\draw [line width=1.5pt]   (C) -- (D);
\draw [line width=1.5pt]  (A) -- (D);
 \draw[->] (2,2) -- (2,3);

\draw[->] (2,4) -- (3,4);

\draw[->] (4,4) -- (4,3);

\draw[->] (2,2) -- (3,2);

\coordinate [label=left:$v_0$]    (A) at (2+6,2);
 \coordinate [label=left:$v_1$]   (B) at (2+6,4);
 \coordinate  [label=right:$v_2$]   (C) at (4+6,4);
\coordinate  [label=right:$v_3$]   (D) at (4+6,2);
\fill (2+6,2) circle (2.5pt);
\fill (2+6,4) circle (2.5pt);
\fill (4+6,4) circle (2.5pt);
\fill (4+6,2) circle (2.5pt);

\coordinate[label=left:$\bar{G}$:] (H) at (7,3);
\draw [line width=1.5pt]  (A) -- (B);
\draw [line width=1.5pt]  (A) -- (C);
\draw [line width=1.5pt]  (B) -- (D);
\draw [line width=1.5pt]  (B) -- (C);
\draw [line width=1.5pt]   (C) -- (D);
\draw [line width=1.5pt]  (A) -- (D);
 \draw[->] (2+6,2) -- (2+6,3);
 \draw[->] (2+6,2) -- (9.5,3.5);

\draw[->] (2+6,4) -- (3+6,4);
\draw[->] (2+6,4) -- (9.5,2.5);

\draw[->] (4+6,4) -- (4+6,3);

\draw[->] (2+6,2) -- (3+6,2);

 \end{tikzpicture}
\end{center}
\caption{Example~\ref{ex-h.17}.}
\end{figure}
\end{example}
\begin{remark}
By Example~\ref{ex-bb} and Example~\ref{ex-h.17}, we know that if a digraph does not satisfy the conditions in Corollary~\ref{co-aa}, then  the isomorphism of homology groups given in (\ref{eq-k1})  may or  not hold.
\end{remark}

 \bigskip

Generally speaking, $\mathrm{Crit}_*(\bar{G})\cap P_*(G)$ in Theorem~\ref{th-aa} and Corollary~\ref{co-aa} can not be replaced by $\mathrm{Crit}_*{\bar{G}}\cap \Omega_*(G)$ which will be illustrated by the following example.
\begin{example}\label{ex-vv}
Consider the digraph given in Example~\ref{ex-w1}. Then
\begin{eqnarray*}
\mathrm{Crit}_*(\bar{G})\cap \Omega_*(G)&=&\{v_1,v_2,v_0v_2,v_2v_3\}.
\end{eqnarray*}
By \cite[Example~3.2]{lin},
\begin{eqnarray*}
&\overline \Phi(v_0)=v_1,&\overline \Phi(v_1)=v_1,\\
&\overline \Phi(v_2)=v_2,
&\overline \Phi(v_3)=v_1,\\
&\overline \Phi(v_0v_1)=0,
&\overline \Phi(v_0v_2)=v_0v_2-v_0v_1,\\
&\overline \Phi(v_1v_3)=0,
&\overline \Phi(v_2v_3)=v_2v_3-v_1v_3,\\
&\overline \Phi(v_0v_3)=0,
&\overline \Phi(v_0v_1v_3)=0,\\
&~~~~~~~~~~~~~~~~~~~~\overline \Phi(v_0v_2v_3)=v_0v_2v_3-v_0v_1v_3
\end{eqnarray*}
where $\overline \Phi=\mathrm{Id}+ \partial \overline V+\overline V\partial$  is the discrete gradient flow of $\bar{G}$. By calculate directly, we have that $\overline \Phi^\infty=\overline \Phi$. Then
\begin{eqnarray*}
P_*^{\overline{\Phi}}(\bar{G})\cap \Omega_*(G)&=&\{v_1,v_2,v_0v_2-v_0v_1,v_2v_3-v_1v_3,v_0v_2v_3-v_0v_1v_3\}.
\end{eqnarray*}
Hence, $\mathrm{Crit}_*(\bar{G})\cap \Omega_*(G)$ and $P_*^{\overline{\Phi}}(\bar{G})\cap \Omega_*(G)$ can not be isomorphic.
Moreover,
\begin{eqnarray*}
\tilde{\partial}(v_0v_2)&=&v_2-v_1,\\
\tilde{\partial}(v_2v_3)&=&v_1-v_2.
\end{eqnarray*}
Therefore,
\begin{eqnarray*}
H_0(\{\mathrm{Crit}(\bar{G})\cap \Omega(G),\tilde{\partial}\})&=&R,\\
 H_1(\{\mathrm{Crit}(\bar{G})\cap \Omega(G),\tilde{\partial}\})&=&R,\\
 H_m(\{\mathrm{Crit}(\bar{G})\cap \Omega(G),\tilde{\partial}\})&=&0 ~\text{for} ~m\geq 2,
\end{eqnarray*}
which are not consistent with the  path homology groups of $G$ given in \cite[Proposition~4.7]{9}.

\end{example}

\section{Morse Inequalities}\label{se-99}
In this section, we will give the Morse inequalities on digraphs by Corollary~\ref{co-aa}.

\bigskip

Given a chain complex
\begin{eqnarray*}
0{\longrightarrow} C_n \overset{\partial_n}{\longrightarrow} C_{n-1}\overset{\partial_{n-1}}{\longrightarrow}\cdots \overset{\partial_1}{\longrightarrow}C_0\longrightarrow 0
\end{eqnarray*}
where $C_p$ is a finite dimensional vector space over $R$ for each $0\leq p\leq n$. Consider its homology
\begin{eqnarray*}
H_p(C_*; R)=\mathrm{Ker}\partial_p /\mathrm{Im}\partial_{p+1}.
\end{eqnarray*}
Then
\begin{eqnarray}
\mathrm{dim}C_p&=&\mathrm{dim}\mathrm{Ker}\partial_p+\mathrm{dim}\mathrm{Im}\partial_p,\nonumber\\
\mathrm {dim}H_p(C_*;R)&=&\mathrm{dim}\mathrm{Ker}\partial_p-\mathrm{dim}\mathrm{Im}\partial_{p+1}\nonumber\\
&=&\mathrm{dim}C_p-\mathrm{dim}\mathrm{Im}\partial_{p}-\mathrm{dim}\mathrm{Im}\partial_{p+1}.\label{eq-99a}
\end{eqnarray}
Hence
\begin{eqnarray}\label{eq-999}
\mathrm{dim}C_p\geq \mathrm {dim}H_p(C_*;R)
\end{eqnarray}
and
\begin{eqnarray*}
&&\big(\mathrm{dim}C_p-\mathrm{dim}H_p(C_*;R)\big)-\big(\mathrm{dim}C_{p-1}-\mathrm{dim}H_{p-1}(C_*;R)\big)+\cdots\nonumber\\
&&+(-1)^p\big(\mathrm{dim}C_{0}-\mathrm{dim}H_{0}(C_*;R)\big)\nonumber\\
&&=\big(\mathrm{dim}\mathrm{Im}\partial_{p}+\mathrm{dim}\mathrm{Im}\partial_{p+1}\big)-
\big(\mathrm{dim}\mathrm{Im}\partial_{p-1}+\mathrm{dim}\mathrm{Im}\partial_{p}\big)+\cdots\nonumber\\
&&+(-1)^p\big(\mathrm{dim}\mathrm{Im}\partial_{0}+\mathrm{dim}\mathrm{Im}\partial_{1}\big)\nonumber\\
&&=\mathrm{dim}\mathrm{Im}\partial_{p+1}+(-1)^p\mathrm{dim}\mathrm{Im}\partial_{0}\nonumber\\
&&=\mathrm{dim}\mathrm{Im}\partial_{p+1}\geq0.
\end{eqnarray*}
Thus,
\begin{eqnarray}
&&\mathrm{dim}C_p-\mathrm{dim}C_{p-1}+\cdots +(-1)^p\mathrm{dim}C_{0}\nonumber\\
&&\geq \mathrm{dim}H_p(C_*;R)-\mathrm{dim}H_{p-1}(C_*;R)+\cdots+(-1)^{p}\mathrm{dim}H_{0}(C_*;R).\label{eq-99b}
\end{eqnarray}

Moreover, by (\ref{eq-99a}), the Euler characteristic of $\{C_*,\partial_*\}$ is
\begin{eqnarray*}
\chi&=&\sum\limits_{p=0}^n(-1)^p\mathrm{dim}H_p(C_*;R)\\
&=&\sum\limits_{p=0}^n(-1)^p\mathrm{dim}C_p.
\end{eqnarray*}
Hence
\begin{eqnarray}
&&\mathrm{dim}C_0-\mathrm{dim}C_1+\cdots+(-1)^n\mathrm{dim}C_n\nonumber\\
&&=\mathrm{dim}H_0(C_*;R)-\mathrm{dim}H_1(C_*;R)+\cdots+(-1)^n\mathrm{dim}H_n(C_*;R)\label{eq-99c}.
\end{eqnarray}

Let $G$ be a digraph and $\bar{G}$ the transitive closure of $G$. Denote
\begin{eqnarray*}
b_m&=&\mathrm {dim}H_m(G;R),\\
l_m&=&\mathrm{dim}\big(\mathrm{Crit}_m(\bar{G})\cap P_m(G)\big),\\
L_m&=&\mathrm{dim}\mathrm{Crit}_m(G).
\end{eqnarray*}

Then we have the following theorem.
\begin{theorem}(Morse Inequalities on Digraphs)\label{th-99}
Let $G$ be a  digraph  and $\bar{G}$ the transitive closure of $G$. Suppose $\Omega_*(G)$ is $\overline{V}$-invariant and $\overline{\Phi}(\alpha)\in \Omega(G)$  for any $\alpha\in \mathrm{Crit}(\bar{G})\cap P(G)$ where $\overline{V}$ is the discrete gradient vector field on $\bar{G}$ and $\overline{\Phi}$ is the discrete gradient flow of $\bar{G}$, respectively. Then
\begin{eqnarray}
&&L_m\geq l_m ,\label{eq-ff}\\
&&l_{m}-l_{m-1}+\cdots\pm l_0\geq b_{m}-b_{m-1}+\cdots\pm b_{0}\label{eq-mm}~ 
\end{eqnarray}
and
\begin{eqnarray}\label{eq-nn}
l_{0}-l_{1}+\cdots\pm l_{\mathrm{dim}G}\geq b_{0}-b_{1}+\cdots\pm b_{\mathrm{dim}G}
\end{eqnarray}
where $m\geq 0$ and $\mathrm{dim}G=\mathrm{max}\{p\mid \Omega_p(G)\not=0\}$.
\end{theorem}

\begin{proof}
By Corollary~\ref{co-aa},
\begin{eqnarray*}
H_m(\{\mathrm{Crit}_*(\bar{G})\cap P_*(G),\tilde{\partial}_*\})\cong H_m(G;R).
\end{eqnarray*}

By substituting $C_m$ and $H_m(C_*;R)$ with $\mathrm{Crit}_m(\bar{G})\cap P_m(G)$ and $H_m(G;R)$ respectively and by (\ref{eq-999}), we  have that
\begin{eqnarray*}\label{eq-88}
l_m\geq b_m.
\end{eqnarray*}
By (\ref{eq-99b}) and (\ref{eq-99c}), we  can obtain (\ref{eq-mm}) and (\ref{eq-nn}) respectively.

Moreover, since $G\subseteq \bar{G}$, by Definition~\ref{def-1.2}, it follows that
\begin{eqnarray*}
\mathrm{Crit}_*(\bar{G})\cap P_*(G)\subseteq \mathrm{Crit}_*(G).
\end{eqnarray*}
Hence  (\ref{eq-ff}) is proved.
\end{proof}
We take the digraphs and discrete Morse functions in Example~\ref{ex-w1} and Example~\ref{ex-aa} as examples to illustrate the above results in Theorem~\ref{th-99}.
\begin{example}
By Example~\ref{ex-w1}, we have that
\begin{eqnarray*}
\mathrm{Crit}_*(\bar{G})\cap P_*(G)&=&\{v_1,v_2,v_0v_2,v_2v_3,v_0v_2v_3\},\\
\mathrm{Crit}_*(G)&=&\{v_1,v_2,v_0v_2,v_2v_3,v_0v_1v_3,v_0v_2v_3\}.
\end{eqnarray*}
and
\begin{eqnarray*}
&&\overline{V}(\Omega(G))\subseteq \Omega(G)\\
&&\overline{\Phi}\big(\mathrm{Crit}_*(\bar{G})\cap P_*(G)\big)\subseteq \Omega_*(G).
\end{eqnarray*}
Then
\begin{eqnarray*}
&&l_0=2,~~ l_1=2, ~~ l_2=1,~~l_i=0~\mathrm{for}~ i>2,\\
&&L_0=2,~~ L_1=2, ~~ L_2=2,~~L_i=0~ \mathrm{for}~ i>2.
\end{eqnarray*}
and
\begin{eqnarray*}
b_0=1,~~b_i=0~\mathrm{for}~i\geq 1.
\end{eqnarray*}
Hence
\begin{eqnarray*}
&&L_m\geq l_m~ \text{for any}~m\geq 0,\\
&&2>1~\text{for}~m=0,\\
&&2-2>0-1 ~\text{for}~m=1,\\
&&1-2+2\geq 0-0+1~\text{for}~m\geq 2,\\
&&\chi(G)=2-2+1=1-0+0=1.
\end{eqnarray*}
Similarly, by Example~\ref{ex-aa}, we have that
\begin{eqnarray*}
&&\mathrm{Crit}_*(\bar{G})\cap P_*(G)=\{v_1,v_2,v_5,v_0v_2,v_2v_3,v_2v_4,v_5v_3,v_5v_4,v_0v_2v_3,v_0v_2v_4\},\\
&&\mathrm{Crit}_*(G)=\{v_1,v_2,v_5,v_0v_2,v_2v_3,v_2v_4,v_5v_3,v_5v_4,v_0v_1v_3,v_0v_2v_3,v_0v_1v_4,v_0v_2v_4\},\\
&&\overline{V}(\Omega(G))\subseteq \Omega(G),\\
&&\overline{\Phi}\big(\mathrm{Crit}_*(\bar{G})\cap P_*(G)\big)\subseteq \Omega_*(G).
\end{eqnarray*}
Then
\begin{eqnarray*}
&&l_0=3,~~ l_1=5, ~~ l_2=2,~~l_i=0~\mathrm{for}~i>2,\\
&&L_0=3,~~ L_1=5, ~~ L_2=4,~~L_i=0~ \mathrm{for}~i>2.
\end{eqnarray*}
and
\begin{eqnarray*}
b_0=1,~~b_1=1,~~b_i=0~\mathrm{for}~i>1.
\end{eqnarray*}
Hence
\begin{eqnarray*}
&&L_m\geq l_m~ \text{for any}~m\geq 0,\\
&&3>1~\text{for}~m=0,\\
&&5-3>1-1~\text{for}~m=1,\\
&&2-5+3\geq 0-1+1~\text{for}~m\geq 2,\\
&&\chi(G)=3-5+2=1-1=0.
\end{eqnarray*}
\end{example}

\bigskip
\noindent {\bf Acknowledgement}. The authors would like to express deep gratitude to the reviewer(s) for their careful reading, valuable comments, and helpful suggestions.

\noindent {$^{*}$ Partially supported by the National Science Foundation of China, Grant No. 12071245.}

\noindent {$^{\dag}$ Partially supported by Natural Fund of Cangzhou Science and Technology Bureau (No.197000002) and a Project of Cangzhou Normal University (No.xnjjl1902).}

Yong Lin

 Address: Yau Mathematical Sciences Center, Tsinghua University, Beijing 100084, China.

 e-mail: yonglin@tsinghua.edu.cn

\medskip

Chong Wang  (for correspondence)

 Address: $^1$School of Mathematics, Renmin University of China, Beijing 100872, China.

 $^2$School of Mathematics and Statistics, Cangzhou Normal University, 061000 China.

 e-mail:  wangchong\_618@163.com

  \medskip

\end{document}